\documentclass[10pt, reqno]{amsart}

\usepackage{color}
\usepackage{hyperref}
\usepackage{graphicx}
\usepackage{amsmath}
\usepackage{amsthm}
\usepackage{amssymb,bbm}
\usepackage{mathrsfs}
\usepackage{enumerate}

\DeclareMathAlphabet{\pazocal}{OMS}{zplm}{m}{n}

\allowdisplaybreaks

  \evensidemargin
\oddsidemargin
\numberwithin{equation}{section}

\theoremstyle{plain}
\newtheorem{theorem}{Theorem}[section]
\newtheorem{proposition}[theorem]{Proposition}
\newtheorem{lemma}[theorem]{Lemma}

\theoremstyle{definition}

\newtheorem{remark}[theorem]{Remark}

\begin{document}
\title[Distribution of Critical Points]{On the Correlation of Critical
Points and Angular Trispectrum for Random Spherical Harmonics}
\author{Valentina Cammarota}
\address{Department of Statistics, Sapienza University of Rome}
\email{valentina.cammarota@uniroma1.it}
\author{Domenico Marinucci}
\address{Department of Mathematics, University of Rome Tor Vergata}
\email{marinucc@mat.uniroma2.it}
\thanks{ This paper has been built out of earlier research carried over with
Igor Wigman; we are very grateful to him for many suggestions and
insightful discussions. DM acknowledges the MIUR Excellence Department
Project awarded to the Department of Mathematics, University of Rome Tor
Vergata, CUP E83C18000100006. VC has received funding from the Istituto
Nazionale di Alta Matematica (INdAM) through the GNAMPA Research
Project 2019 ``Propriet\`a analitiche e geometriche di campi aleatori".}

\begin{abstract}
We prove a Central Limit Theorem for the critical points of random spherical
harmonics, in the high-energy limit. The result is a consequence of a deeper
characterization of the total number of critical points, which are shown to
be asymptotically fully correlated with the sample trispectrum, i.e., the
integral of the fourth Hermite polynomial evaluated on the eigenfunctions
themselves. As a consequence, the total number of critical points and the
nodal length are fully correlated for random spherical harmonics, in the
high-energy limit.
\end{abstract}

\maketitle

\begin{itemize}
\item \textbf{AMS Classification}: 60G60, 62M15, 53C65, 42C10, 33C55.

\item \textbf{Keywords and Phrases}: {Critical Points, Wiener-Chaos
Expansion, Spherical Harmonics, Correlation, Berry's Cancellation Phenomenon}
\end{itemize}

\section{Introduction and Main Results}

\subsection{Random spherical harmonics and sample polyspectra}

It is well-known that the eigenvalues \\ $\left\{ -\lambda _{\ell }\right\}
_{\ell =0, 1, 2, \dots }$ of the Helmholtz equation
\begin{equation*}
\Delta _{\mathbb{S}^{2}}f+\lambda _{\ell }f=0,\hspace{1cm} \Delta _{\mathbb{S}
^{2}}=\frac{\partial ^{2}}{\partial \theta ^{2}}+\cot \theta \frac{\partial
}{\partial \varphi }+\frac{1}{\sin ^{2}\theta }\frac{\partial ^{2}}{\partial
\varphi ^{2}},\hspace{1cm} \ell =1, 2, \dots .
\end{equation*}
on the two-dimensional sphere $\mathbb{S}^{2}$, are of the form
$\lambda _{\ell }=\ell (\ell +1)$ for some integer $\ell \geq 1$. For any
given eigenvalue $-\lambda _{\ell }$, the corresponding eigenspace is the $
(2\ell +1)$-dimensional space of spherical harmonics
of degree $\ell $; we can choose an arbitrary $L^{2}$-orthonormal basis $
\left\{ Y_{\mathbb{\ell }m}(.)\right\} _{m=-\ell ,\dots ,\ell }$, and
consider random eigenfunctions of the form
\begin{equation*}
f_{\ell }(x)=\frac{\sqrt{4\pi }}{\sqrt{2\ell +1}}\sum_{m=-\ell }^{\ell
}a_{\ell m}Y_{\ell m}(x),
\end{equation*}
where the coefficients $\left\{ a_{\mathbb{\ell }m}\right\} $ are
independent, standard Gaussian variables if the basis is chosen to be
real-valued; the standardization is such that $\text{Var}(f_{\ell }(x))=1$,
and the representation is invariant with respect to the choice of any
specific basis $\left\{ Y_{\ell m},\text{ }m=-\ell ,\dots ,\ell \right\} $. The
random fields $\{f_{\ell }(x),\;x\in \mathbb{S}^{2}\}$ are isotropic,
meaning that the probability laws of $f_{\ell }(\cdot )$ and $f_{\ell
}^{g}(\cdot ):=f_{\ell }(g\cdot )$ are the same for any rotation $g\in SO(3)$;
 they are also centred and Gaussian, and from the addition theorem for
spherical harmonics (see \cite{MaPeCUP}, Equation (3.42)) the covariance
function is given by,
\begin{equation*}
\mathbb{E}[f_{\ell }(x)f_{\ell }(y)]=P_{\ell }(\cos d(x,y)),
\end{equation*}
where $P_{\ell }$ are the usual Legendre polynomials, $\cos d(x,y)=\cos
\theta _{x}\cos \theta _{y}+\sin \theta _{x}\sin \theta _{y}\cos (\varphi
_{x}-\varphi _{y})$ is the spherical geodesic distance between $x$ and $y$, $
\theta \in \lbrack 0,\pi ]$, $\varphi \in \lbrack 0,2\pi )$ are standard
spherical coordinates and $(\theta _{x},\varphi _{x})$, $(\theta
_{y},\varphi _{y})$ are the spherical coordinates of $x$ and $y,$
respectively.\\

\noindent In this paper, we shall be concerned with the number of critical points of
$f_{\ell }(\cdot )$, defined as usual as
\begin{equation*}
\mathcal{N}_{\ell }^{c}=\left\{ x\in \mathbb{S}^{2}:\nabla f_{\ell
}(x)=0\right\};
\end{equation*}
it was shown in \cite{nicolaescu} (see also \cite{CMW}) that we have
\begin{equation*}
\mathbb{E}[\mathcal{N}_{\ell }^{c}]=\frac{2}{\sqrt{3}}\ell (\ell +1)+O(1),
\end{equation*}
whereas (see \cite{CW}) the variance of $\mathcal{N}_{\ell }^{c}$ is
asymptotic to
\begin{equation}
\text{Var}(\mathcal{N}_{\ell }^{c})=\frac{\ell ^{2}\log \ell }{3^{3}\pi ^{2}}
+O(\ell ^{2}), \text{ as }\ell \rightarrow \infty.  \label{varcp}
\end{equation}
We now study the limiting distribution of the fluctuations around the
expected value. First recall that the sequence of Hermite polynomials
$H_{q}(u)$ is defined by
\begin{equation*}
H_{q}(u):=(-1)^{q}\frac{1}{\phi (u)}\frac{d^{q}\phi (u)}{du^{q}}, \hspace{1cm}
\phi (u):=\frac{1}{\sqrt{2\pi }}\exp\left\{-\frac{u^{2}}{2}\right\},
\end{equation*}
so that
\begin{equation*}
H_{0}(u)=1, \;\; H_{1}(u)=u, H_{2}(u)=u^{2}-1, \;\;
H_{3}(u)=u^{3}-3u, \;\; H_{4}(u)=u^{4}-6u^{2}+3, \dots
\end{equation*}
We refer to \cite{noupebook} for a detailed discussion of Hermite
polynomials, their properties and their ubiquitous role in the analysis of
Gaussian processes. Below we shall also exploit the sequence of (random)
sample polyspectra, which we define as (see, e.g., \cite{CMW}, \cite{MW},
\cite{MW2014}, \cite{PTRF09}, \cite{MR2015})
\begin{equation*}
h_{\ell ;q}:=\int_{\mathbb{S}^{2}}H_{q}(f_{\ell }(x))dx.
\end{equation*}
It is readily checked that $h_{\ell ;0}=4\pi $ and $h_{\ell ;1}=0,$ for all $\ell$;
we also have $\mathbb{E}\left[ h_{\ell ;q}\right] =0$, for all
$q=1,2,\dots$ As far as variances are concerned, we have that (see \cite{MW},
\cite{MW2014})
\begin{equation*}
\text{Var} \left( h_{\ell ;2}\right) =(4\pi )^{2}\frac{2}{2\ell +1}, \hspace{0.7cm}
\text{Var}\left( h_{\ell ;4}\right) =\frac{576\log \ell }{\ell ^{2}}+O(\ell^{-2}), \hspace{0.7cm}
\text{Var}\left( h_{\ell ;q}\right) =\frac{c_{q}}{\ell ^{2}}+o(\ell^{-2}),
\end{equation*}
for $q=3, 5, 6\dots$, where
\begin{equation*}
c_{q}:=\int_{0}^{\infty }J_{0}(\psi )^{q}\psi d\psi, \hspace{1cm} J_{0}(\psi
)=\sum_{k=0}^{\infty }\frac{(-1)^{k}x^{2k}}{(k!)^{2}2^{2k}},
\end{equation*}
and $J_{0}(.)$ is the usual Bessel function of the first kind.

\subsection{Main results}

Our first main result in this paper is to show that the number of critical
points and the sample trispectrum $\left\{ h_{\ell ;4}\right\} $ are
asymptotically fully correlated: as $\ell \rightarrow \infty$
\begin{equation}
\lim_{\ell \rightarrow \infty }\rho^{2}(\mathcal{N}_{\ell }^{c},h_{\ell
;4}):=\lim_{\ell \rightarrow \infty }\frac{\text{Cov}^{2}(\mathcal{N}_{\ell
}^{c},h_{\ell ;4})}{\text{Var}(\mathcal{N}_{\ell }^{c})\text{Var}(h_{\ell ;4})}=1.
\label{giugno}
\end{equation}
In fact, our result is sharper than that. Recall first that the variance for
the total number of critical points was computed in \cite{CW} to be
asymptotic to
\begin{equation*}
\text{Var}(\mathcal{N}_{\ell }^{c})=\frac{\ell ^{2}\log \ell }{3^{3}\pi ^{2}}
+O(\ell ^{2}), \;\; \text{ as }\ell \rightarrow 0.
\end{equation*}
Let us now introduce the random sequence
\begin{equation*}
\mathcal{A}_{\ell }=-\frac{\lambda _{\ell }}{2^{3}3^{2}\sqrt{3}\pi }\int_{
\mathbb{S}^{2}}H_{4}(f_{\ell }(x))dx=-\frac{\lambda _{\ell }}{2^{3}3^{2}
\sqrt{3}\pi }h_{\ell ;4},
\end{equation*}
for which it is readily seen that
\begin{equation*}
\mathbb{E}\left[ \mathcal{A}_{\ell }\right] =0, \;\;\;\; \lim_{\ell
\rightarrow \infty }\frac{\text{Var}(\mathcal{N}_{\ell }^{c})}{\text{Var}(\mathcal{A}
_{\ell })}=1,
\end{equation*}
because
\begin{equation*}
\text{Var}(\mathcal{A}_{\ell }) =\frac{\lambda _{\ell }^{2}}{2^{6}3^{5}\pi ^{2}}
\text{Var}(h_{\ell ;4})= \frac{\lambda _{\ell }^{2}}{2^{6}3^{5}\pi ^{2}}\left\{ \frac{576\log \ell
}{\ell ^{2}}+O(\ell^{-2})\right\} =\frac{\ell ^{2}\log \ell }{
3^{3}\pi ^{2}}+O(\ell ^{2}), \;\; \text{ as }\ell \rightarrow 0.
\end{equation*}
It is convenient to write
\begin{equation*}
\widetilde{\mathcal{A}}_{\ell }=\frac{\mathcal{A}_{\ell }}{\sqrt{\text{Var}(
\mathcal{A}_{\ell })}}.
\end{equation*}
We can now formulate the following
\begin{theorem}
\label{tisone1} As $\ell \rightarrow \infty $
\begin{equation*}
\rho(\mathcal{N}_{\ell }^{c},\mathcal{A}_{\ell })=\frac{\rm{Cov}(\mathcal{N}
_{\ell }^{c},\mathcal{A}_{\ell })}{\sqrt{\rm{Var}(\mathcal{N}_{\ell }^{c})\rm{Var}(
\mathcal{A}_{\ell })}}\rightarrow 1,
\end{equation*}
and hence
\begin{equation*}
\frac{\mathcal{N}_{\ell }^{c}-\mathbb{E}\left[ \mathcal{N}_{\ell }^{c}\right]
}{\sqrt{\rm{Var}(\mathcal{N}_{\ell }^{c})}}=\widetilde{\mathcal{A}}_{\ell
}+o_{p}(1).
\end{equation*}
\end{theorem}
As a consequence of the previous theorem, for $\ell \rightarrow \infty$, we
have that \eqref{giugno} holds, so that the total number of critical points is
fully correlated in the limit with $\left\{ h_{\ell ;4}\right\}$. The
limiting distribution of $\left\{ h_{\ell ;4}\right\} $ was already studied
in \cite{MW2014}, where it was shown that a (quantitative version of the)
Central Limit Theorem holds. Our next main result hence follows immediately;
recall first that the Wasserstein distance between the probability
distributions of two random variables $(X,Y)$ is defined by
\begin{equation*}
d_{W}(X,Y):=\sup_{h\in {\rm Lip}(1)}|\mathbb{E}h(X)-\mathbb{E}h(Y)|,
\end{equation*}
where
\begin{equation*}
{\rm Lip} (1):=\left\{ h:\mathbb{R}\rightarrow \mathbb{R}\text{ such that }
\left\vert \frac{h(x)-h(y)}{x-y}\right\vert \leq 1\text{ for all }x\neq
y\right\}.
\end{equation*}
\begin{theorem}
\label{tisone2} As $\ell \rightarrow \infty ,$ for $Z$ a
standard Gaussian variable, we have that
\begin{equation*}
\lim_{\ell \rightarrow \infty }d_{W}\left(\frac{\mathcal{N}_{\ell }^{c}-\mathbb{E}
\left[ \mathcal{N}_{\ell }^{c}\right] }{\sqrt{{\rm Var}\left( \mathcal{N}_{\ell
}^{c}\right) }},Z\right)=0,
\end{equation*}
and hence
\begin{equation*}
\frac{\mathcal{N}_{\ell }^{c}-\mathbb{E}\left[ \mathcal{N}_{\ell }^{c}\right]
}{\sqrt{{\rm Var}\left( \mathcal{N}_{\ell }^{c}\right) }}\rightarrow _{d}Z.
\end{equation*}
\end{theorem}
\begin{proof}
It was shown in \cite{MW2014} that
\begin{equation*}
\lim_{\ell \rightarrow \infty }d_{W}\left(\frac{h_{\ell ;4}}{\sqrt{{\rm Var}\left(
h_{\ell ;4}\right) }},Z\right)=0;
\end{equation*}
the result then follows from Theorem \ref{tisone1} and the triangle
inequality
\begin{align*}
\lim_{\ell \rightarrow \infty }d_{W}\left(\frac{\mathcal{N}_{\ell }^{c}-\mathbb{E}
\left[ \mathcal{N}_{\ell }^{c}\right] }{\sqrt{\text{Var}\left( \mathcal{N}_{\ell
}^{c}\right) }},Z\right) &\leq \lim_{\ell \rightarrow \infty }d_{W}\left(\frac{
\mathcal{N}_{\ell }^{c}-\mathbb{E}\left[ \mathcal{N}_{\ell }^{c}\right] }{
\sqrt{\text{Var}\left( \mathcal{N}_{\ell }^{c}\right) }},\frac{\mathcal{A}_{\ell }
}{\sqrt{\text{Var}\left( \mathcal{A}_{\ell }\right) }}\right)+\lim_{\ell \rightarrow
\infty }d_{W}\left(\frac{\mathcal{A}_{\ell }}{\sqrt{\text{Var}\left( \mathcal{A}_{\ell
}\right) }},Z\right) \\
&=\lim_{\ell \rightarrow \infty }d_{W}\left(\frac{\mathcal{N}_{\ell }^{c}-
\mathbb{E}\left[ \mathcal{N}_{\ell }^{c}\right] }{\sqrt{\text{Var}\left( \mathcal{N
}_{\ell }^{c}\right) }},\frac{\mathcal{A}_{\ell }}{\sqrt{\text{Var}\left(
\mathcal{A}_{\ell }\right) }}\right)+\lim_{\ell \rightarrow \infty }d_{W}\left(\frac{
h_{\ell ;4}}{\sqrt{\text{Var}\left( h_{\ell ;4}\right) }},Z\right) \\
&=0.
\end{align*}
\end{proof}
\begin{remark}
The previous theorems include actually two separate results, namely:
\begin{enumerate}[i)]
\item  the asymptotic behaviour of the total number of critical points is
dominated by its projection on the fourth-order chaos term (see Section \ref{caos});

\item the projection on the fourth-order chaos can be expressed simply in terms
of the fourth-order Hermite polynomial, evaluated on the eigenfunctions $\left\{
f_{\ell }\right\}$, without the need to compute Hermite polynomials
evaluated on the first and second derivatives of $\left\{f_{\ell }\right\}$, despite the fact that the
latter do appear in the Kac-Rice formula and they are not negligible in
terms of asymptotic variance.
\end{enumerate}

As we shall discuss in the following section, both these findings have
analogous counterparts in the behaviour of the boundary and nodal
length, as investigated, i.e., in \cite{MRW}.
\end{remark}

\subsection{Discussion: correlation between critical points and nodal
length.}

The results in our paper should be compared with a recent stream of
literature which has investigated the relationship between geometric
features of random spherical harmonics and sample polyspectra. The first
results in this area are due to \cite{MW}, which studied the excursion area
of $\left\{ f_{\ell }\right\}$ above a threshold
$u\in \mathbb{R}$ (which we label $\mathcal{L}_{2}(u;\ell
) $), and showed that it is asymptotically dominated (after centering) by a
term of the form $-u\phi (u)h_{\ell ;2}/2;$ in particular, they showed that
\begin{enumerate}[i)]
\item  there is full correlation, in the high-energy limit, between $h_{\ell ;2}$
and the excursion area, for all $u\neq 0$;

\item for $u=0$ (the case of the so-called Defect) this leading term vanishes,
and the asymptotic behaviour is radically different: all the odd-order
chaoses of order greater or equal to 3 are correlated with the excursion
area.
\end{enumerate}
The same pattern of behaviour was later established for the boundary length
$\mathcal{L}_{1}(u;\ell )$ (for $u\neq 0)$ and the Euler
characteristic $\mathcal{L}_{0}(u;\ell )$ (for $u\neq 0,\pm 1,$ see i.e.,
\cite{CM2018} and the references therein), thus covering the behaviour of
all three Lipschitz-Killing Curvatures (see \cite{adlertaylor}). More
explicitly, we have that, as $\ell \rightarrow \infty$, (see i.e., \cite
{CM2018})
\begin{align*}
\mathcal{L}_{0}(u;\ell ) - \mathbb{E}[\mathcal{L}_{0}(u;\ell )] &=\frac{1}{2}\frac{\lambda _{\ell }}{2}
H_{2}(u)H_{1}(u)\phi (u)\frac{1}{2\pi }h_{\ell ;2}+o_{p}(\ell ^{3/2}) \\
&=\frac{1}{2}\frac{\lambda _{\ell }}{2}(u^{3}-u)\phi (u)\frac{1}{2\pi }
h_{\ell ;2}+o_{p}(\ell ^{3/2}), \\
\mathcal{L}_{1}(u;\ell ) - \mathbb{E} [\mathcal{L}_{1}(u;\ell )] &=\frac{1}{2}\sqrt{\frac{\lambda _{\ell }}{2}}
\sqrt{\frac{\pi }{8}}H_{1}^{2}(u)\phi (u)h_{\ell ;2}+o_{p}(\ell^{1/2}) \\
&=\frac{1}{2}\sqrt{\frac{\lambda _{\ell }}{2}}\sqrt{\frac{\pi }{8}}
u^{2}\phi (u)h_{\ell ;2}+o_{p}(\ell^{1/2}), \\
\mathcal{L}_{2}(u;\ell ) - \mathbb{E}[\mathcal{L}_{2}(u;\ell )]&= \frac{1}{2}H_{1}(u)\phi (u)h_{\ell;2}+o_{p}(\ell^{-1/2}) \\
&=\frac{1}{2}u\phi (u)h_{\ell ;2}+o_{p}(\ell^{-1/2}),
\end{align*}
whence
\begin{equation*}
\lim_{\ell \rightarrow \infty }\rho^{2}(h_{\ell ;2};\mathcal{L}_{2}(u;\ell
)) =\lim_{\ell \rightarrow \infty }\rho ^{2}(h_{\ell ;2};\mathcal{L}
_{1}(u;\ell ))=1, \text{ for }u\neq 0,
\end{equation*}
\begin{equation*}
\lim_{\ell \rightarrow \infty } \rho^{2}(h_{\ell ;2};\mathcal{L}_{0}(u;\ell
)) =1,\text{ for }u\neq 0,\pm 1,
\end{equation*}
and
\begin{equation*}
\lim_{\ell \rightarrow \infty }\rho ^{2}(\mathcal{L}_{0}(u;\ell ),\mathcal{L}
_{1}(u;\ell ))=\lim_{\ell \rightarrow \infty }\rho ^{2}(\mathcal{L}
_{0}(u;\ell ),\mathcal{L}_{2}(u;\ell ))=\lim_{\ell \rightarrow \infty }\rho
^{2}(\mathcal{L}_{1}(u;\ell ),\mathcal{L}_{2}(u;\ell ))=1, \text{ for }u\neq
0,\pm 1.
\end{equation*}
Loosely speaking, it can be concluded that these three Lipschitz-Killing
curvatures are asymptotically proportional to $h_{\ell ;2}$ in the
high-energy limit for $u\neq 0$ (and also for $u\neq \pm 1$ in the case of Euler
characteristics), and thus they are fully correlated at different thresholds
and among themselves.

These results were extended in \cite{CM2018b} to critical values over the
interval $I$. More precisely, let $I\subseteq \mathbb{R}$ be any interval in
the real line; we are interested in the number of critical points of $
f_{\ell }$ with value in $I$:
\begin{equation*}
\mathcal{N}_{\ell }^{c}(I)=\#\{x\in \mathbb{S}^{2}:f_{\ell }(x)\in I,\nabla
f_{\ell }(x)=0\}.
\end{equation*}
For the expectation, it was shown in \cite{CMW} that for every interval
$I\subseteq \mathbb{R}$ we have, as $\ell \rightarrow \infty$,
\begin{equation*}
\mathbb{E}[\mathcal{N}_{\ell }^{c}(I)]=\frac{2}{\sqrt{3}}\lambda _{\ell
}\int_{I}\pi _{1}^{c}(t)dt+O(1), \hspace{1cm} \pi _{1}^{c}(t)=\frac{\sqrt{3}}{
\sqrt{8\pi }}(2e^{-t^{2}}+t^{2}-1)e^{-\frac{t^{2}}{2}};
\end{equation*}
moreover, for $I$ such that
\begin{equation*}
\nu ^{c}(I):=\left[ \int_{I}p_{3}^{c}(t)dt\right] ^{2}\neq 0, \hspace{1cm}
p_{3}^{c}(t)=\frac{1}{\sqrt{8\pi }}e^{-\frac{3}{2}t^{2}}\left[
2-6t^{2}-e^{t^{2}}(1-4t^{2}+t^{4})\right],
\end{equation*}
we have that
\begin{equation*}
\lim_{\ell \rightarrow \infty }\rho ^{2}(h_{\ell ;2};\mathcal{N}_{\ell
}^{c}(I))=1.
\end{equation*}
More precisely, it was shown in \cite{CM2018b}, that
\begin{equation*}
\mathcal{N}_{\ell }^{c}(I)-\mathbb{E}[\mathcal{N}_{\ell }^{c}(I)]=\left[ \ell ^{3/2}\int_{I}p_{3}^{c}(t)dt\right]
\times \frac{h_{\ell ;2}}{\sqrt{\text{Var}(h_{\ell ;2})}}+o_{p}\left(\sqrt{\text{Var}(\mathcal{N
}_{\ell }^{c}(I))}\right), \text{ as }\ell \rightarrow \infty.
\end{equation*}
We call $I$ nondegenerate if and only if
$$\int_I p_3^c(t) d t \ne 0.$$
For instance semi-intervals $I=[u,\infty )$ with $u\neq 0$ are nondegenerate. As a
consequence, for the same range of values of $u$, we have that
\begin{equation*}
\lim_{\ell \rightarrow \infty }\rho ^{2}(\mathcal{L}_{a}(u;\ell ),\mathcal{N}
_{\ell }^{c}([u,\infty )))=1, \text{ for }a=0,1,2.
\end{equation*}
For $I=[0,\infty )$ or $\mathbb{R}$ (corresponding to the total number of
critical points), the leading constant $\nu ^{c}(I)$ vanishes, and,
accordingly, the order of magnitude of the variance is smaller than $\ell
^{3}$; indeed, as $\ell \rightarrow \infty $ (see \cite{CW}),
\begin{equation*}
\text{Var}(\mathcal{N}_{\ell }^{c})=\frac{1}{3^{3}\pi ^{2}}\ell ^{2}\log
\ell +O(\ell ^{2}).
\end{equation*}
This behaviour is again similar to what was found for $\mathcal{L}
_{1}(0;\ell )$ (the {\it nodal length} of random spherical harmonics), for which
it was shown in \cite{Wig} that
\begin{equation*}
\text{Var}(\mathcal{L}_{1}(0;\ell ))=\frac{1}{128}\log \ell +O(1);
\end{equation*}
actually our expression here differs from the one in \cite{Wig} by a factor $
{1}/{4}$, because $\mathcal{L}_{1}(0;\ell )$ is equivalent to half the
{\it nodal length} of random spherical harmonics considered in that paper. It
was later shown in \cite{MRW} that the following asymptotic equivalence
holds:
\begin{equation*}
\mathcal{L}_{1}(0;\ell )- \mathbb{E}[\mathcal{L}_{1}(0;\ell )] =-\frac{1}{4}\sqrt{\frac{\lambda _{\ell }}{2}}\frac{1}{4!}h_{\ell ;4}+o_{p}(\sqrt{\log \ell }),
\end{equation*}
consistent with the computation of the variance in \cite{Wig}, because (see
\cite{MW2014})
\begin{equation*}
\text{Var}(h_{\ell ;4})=\frac{576\log \ell }{\ell ^{2}}+O(\ell ^{-2}), \text{ as } \ell \rightarrow \infty.
\end{equation*}
In particular, we have that
\begin{equation*}
\lim_{\ell \rightarrow \infty }\rho ^{2}(\mathcal{L}_{1}(0;\ell );h_{\ell;4})=1.
\end{equation*}
Our results in this paper show that the asymptotic behaviour of the total
number of critical points (i.e. $I=\mathbb{R}$) is dominated by
exactly the same component as the nodal length, and indeed
\begin{equation*}
\lim_{\ell \rightarrow \infty }\rho ^{2}(\mathcal{L}_{1}(0;\ell );\mathcal{N}
_{\ell }^{c})=\lim_{\ell \rightarrow \infty }\rho ^{2}(\mathcal{N}_{\ell
}^{c};h_{\ell ;4})=1.
\end{equation*}
Summing up, the literature so far has established the full correlation of
Lipschitz-Killing curvatures and critical values among themselves and with
the sequence $\left\{ h_{\ell ;2}\right\} $ for nondegenerate values of the
threshold parameter $u$. Here we show that in the degenerate cases ($
u=-\infty, 0$ for critical points) full-correlation still exists between
nodal length and critical points, as both are proportional to the sample
trispectrum $h_{\ell ;4}=\int_{\mathbb{S}^{2}}H_{4}(f_{\ell }(x))dx$. The
correlation is positive, which is to say that the realization that
corresponds to a higher number of critical points are those where longer
nodal lines are going to be observed. Heuristically, it can be conjectured
that a higher number of critical points will typically correspond to a
higher number of nodal components, and hence nodal length will be as well
larger than average. One cautious note is needed here: whereas the
correlation converges to unity, it does so only at a logarithmic rate, so it
may not be simple to visualize this effect by simulations with values of $
\ell $ in the order of a few hundreds. On the contrary, the correlation for
values of the threshold $u$ different from zero occurs with rate $\ell ^{-1}$
and shows up very neatly in simulations.

A number of other papers have investigated the geometry of random
eigenfunctions on the sphere and on the torus in the last few years. Among
these, we recall \cite{MW} and \cite{MW2014} for the excursion area and the
Defect; \cite{KKW}, \cite{MPRW2015}, \cite{Cammarota2018} and \cite{BenatarMaffucci} for the nodal
length/volume of arithmetic random waves; \cite{dalmao} for the number
of intersections of random eigenfunctions; \cite{MRW} for the nodal length
of random spherical harmonics; \cite{npr} for the nodal length of Berry's
random waves on the plane; \cite{Rudnick} and \cite{RudnickYesha} for nodal
intersections; \cite{BMW}, \cite{Todino1} and \cite{Todino2} for fluctuations
over subsets of the torus and of the sphere. Zeroes of
random trigonometric polynomials have been considered, for instance, by \cite{angstdalmao}, \cite{angstuniv}, \cite{bally17} and the references therein.

\subsection{Plan of the paper}

In Section \ref{KacRice} we present some background material on Kac-Rice
techniques, Wiener chaos expansions and the relevant covariant matrices for
our (covariant) gradient and Hessian. The proof of our main result is given
in Section \ref{DimMain}, where we show that the total number of critical
points is asymptotically fully correlated with the integral of the
fourth-Hermite polynomial evaluated on the eigenfunctions themselves.
The projection coefficients on Wiener chaoses that we shall need are only three, and their computation is
collected in Section \ref{EvaCoeff}.  In Section \ref{oddt} we consider the terms in the chaos expansion with odd index Hermite polynomials.
The technical computations are in Appendix \ref{Appendix}.

\section{Kac-Rice Formula and the Chaos Expansion $\label{KacRice}$} \label{caos}

As discussed in \cite{CMW}, \cite{CMW-EPC}, \cite{CM2018} and \cite{CM2018b},
by means of Kac-Rice formula, the number of critical points can be
formally written as
\begin{equation*}
\mathcal{N}_{\ell }^{c}=\int_{\mathbb{S}^{2}}|\text{det}\nabla ^{2}f_{\ell}(x)|
{\boldsymbol \delta} (\nabla f_{\ell }(x))dx,
\end{equation*}
where the identity holds both almost surely (using i.e., the Federer's
coarea formula, see \cite{adlertaylor}), and in the $L^{2}$ sense, i.e.,
\begin{equation*}
\mathcal{N}_{\ell }^{c} = \lim_{\varepsilon \rightarrow 0 } \int_{
\mathbb{S}^{2}}|\text{det}\nabla ^{2}f_{\ell}(x)| {\boldsymbol \delta}_{\varepsilon }
(\nabla f_{\ell }(x))dx =\lim_{\varepsilon \rightarrow 0 }\mathcal{N}_{\ell ,\varepsilon}^{c}
\end{equation*}
for
\begin{equation*}
\mathcal{N}_{\ell ,\varepsilon }^{c}:=\int_{
\mathbb{S}^{2}}|\text{det}\nabla ^{2}f_{\ell }(x)| {\boldsymbol \delta}_{\varepsilon
}(\nabla f_{\ell }(x))dx, \hspace{1cm} {\boldsymbol \delta}_{\varepsilon }(.):=\frac{1}{
(2\varepsilon)^2 }\mathbb{I}_{[-\varepsilon ,\varepsilon ]^2}( \cdot , \cdot ).
\end{equation*}
The validity of this limit, in the $L^{2}(\Omega )$ sense, was shown in \cite
{CM2018}, \cite{CM2018b}. The approach for the proof is to start from
the Wiener chaos expansion
\begin{equation}
\mathcal{N}_{\ell }^{c}=\sum_{q=0}^{\infty }\text{Proj}[\mathcal{N}_{\ell
}^{c}|q]=:\sum_{q=0}^{\infty }\mathcal{N}_{\ell }^{c}[q],
\label{chaosexp}
\end{equation}
where $\left\{ \mathcal{N}_{\ell }^{c}[q]\right\} $ denotes the \emph{chaos-component of order q},
or equivalently the projection of $\mathcal{N}_{\ell }^{c}$ on the $q$th order chaos components, which we shall describe
below. In order to define and compute more explicitly these chaos components, let
us introduce the differential operators
\begin{align*}
\partial _{1;x} =\left. \frac{\partial }{\partial \theta }\right\vert
_{\theta =\theta _{x},\varphi =\varphi _{x}}, \hspace{1cm} \partial _{2;x}=\left.
\frac{1}{\sin \theta }\frac{\partial }{\partial \varphi }\right\vert
_{\theta =\theta _{x},\varphi =\varphi _{x}},
\end{align*}
\begin{align*}
\partial _{11;x} =\left. \frac{\partial ^{2}}{\partial \theta ^{2}}
\right\vert _{\theta =\theta _{x},\varphi =\varphi _{x}}, \hspace{1cm} \partial
_{12;x}=\left. \frac{1}{\sin \theta }\frac{\partial^{2}}{\partial \theta
\partial \varphi }\right\vert_{\theta =\theta _{x}, \varphi =\varphi _{x}},\hspace{1cm}
\partial _{22;x}=\left. \frac{1}{\sin ^{2}\theta }\frac{\partial ^{2}
}{\partial \varphi ^{2}}\right\vert _{\theta =\theta _{x},\varphi =\varphi
_{x}}.
\end{align*}
Covariant gradient and Hessian follow the standard definitions, discussed
for instance in \cite{CM2018}; here we simply recall that
\begin{align*}
\nabla f_{\ell }(x) &=(\partial_{1}f_{\ell }(x),\partial_{2}f_{\ell }(x)), \\
\nabla ^{2}f_{\ell }(x) &=\left(
\begin{array}{cc}
\partial _{11}f_{\ell }(x) & \partial _{12}f_{\ell }(x)-\cot \theta
_{x}\partial _{2}f_{\ell }(x) \\
\partial_{12}f_{\ell }(x)-\cot \theta _{x}\partial_{2}f_{\ell }(x) &
\partial_{22}f_{\ell }(x)+\cot \theta_{x}\partial_{1}f_{\ell }(x)
\end{array}
\right), \\
vec \nabla ^{2}f_{\ell }(x) &=\left(\partial _{11}f_{\ell }(x),\partial
_{12}f_{\ell }(x)-\cot \theta _{x}\partial _{2}f_{\ell }(x),\partial
_{22}f_{\ell }(x)+\cot \theta _{x}\partial _{1}f_{\ell }(x)\right).
\end{align*}
We can then introduce the $5\times 1$ vector $(\nabla f_{\ell }(x),vec\nabla
^{2}f_{\ell }(x))$; its covariance matrix $\sigma _{\ell }$ is constant with respect to $x$
and it is computed in \cite{CM2018b}. It can be written in the partitioned form
\begin{equation*}
\sigma _{\ell}=\left(
\begin{array}{cc}
a_{\ell } & b_{\ell } \\
b_{\ell }^{T} & c_{\ell }
\end{array}
\right),
\end{equation*}
where the superscript $T$ denotes transposition, and
\begin{equation*}
a_{\ell }=\left(
\begin{array}{cc}
\frac{\lambda _{\ell }}{2} & 0 \\
0 & \frac{\lambda _{\ell }}{2}
\end{array}
\right) ,\hspace{1cm}b_{\ell }=\left(
\begin{array}{ccc}
0 & 0 & 0 \\
0 & 0 & 0
\end{array}
\right),\hspace{1cm}
c_{\ell }=\frac{\lambda _{\ell }^{2}}{8}\left(
\begin{array}{ccc}
3-\frac{2}{\lambda _{\ell }} & 0 & 1+\frac{2}{\lambda _{\ell }} \\
0 & 1-\frac{2}{\lambda _{\ell }} & 0 \\
1+\frac{2}{\lambda _{\ell }} & 0 & 3-\frac{2}{\lambda _{\ell }}
\end{array}
\right).
\end{equation*}
\noindent Let us recall that the \emph{Cholesky decomposition} of a
Hermitian positive-definite matrix $A$ takes the form $A=\Lambda \Lambda
^{T},$ where $\Lambda $ is a lower triangular matrix with real and positive
diagonal entries, and $\Lambda ^{T}$ denotes the conjugate transpose of $
\Lambda $. It is well-known that every Hermitian positive-definite matrix
(and thus also every real-valued symmetric positive-definite matrix) admits
a unique Cholesky decomposition.

By an explicit computation, it is possible to show that the Cholesky
decomposition of $\sigma _{\ell }$ takes the form $\sigma _{\ell }=\Lambda
_{\ell }\Lambda _{\ell }^{t},$ where
\begin{equation*}
\Lambda _{\ell }=\left(
\begin{array}{ccccc}
\frac{\sqrt{\lambda }_{\ell }}{\sqrt{2}} & 0 & 0 & 0 & 0 \\
0 & \frac{\sqrt{\lambda }_{\ell }}{\sqrt{2}} & 0 & 0 & 0 \\
0 & 0 & \frac{\sqrt{\lambda _{\ell }}\sqrt{3\lambda _{\ell }-2}}{2\sqrt{2}}
& 0 & 0 \\
0 & 0 & 0 & \frac{\sqrt{\lambda _{\ell }}\sqrt{\lambda _{\ell }-2}}{2\sqrt{2}
} & 0 \\
0 & 0 & \frac{\sqrt{\lambda _{\ell }}(\lambda _{\ell }+2)}{2\sqrt{2}\sqrt{
3\lambda _{\ell }-2}} & 0 & \frac{\lambda _{\ell }\sqrt{{\lambda _{\ell }-2}}
}{\sqrt{3\lambda _{\ell }-2}}
\end{array}
\right) =:\left(
\begin{array}{ccccc}
\tau _{1} & 0 & 0 & 0 & 0 \\
0 & \tau _{1} & 0 & 0 & 0 \\
0 & 0 & \tau _{3} & 0 & 0 \\
0 & 0 & 0 & \tau _{4} & 0 \\
0 & 0 & \tau _{2} & 0 & \tau _{5}
\end{array}
\right);
\end{equation*}
in the last expression, for notational simplicity we have omitted the
dependence of the $\tau _{i}$s on $\ell $. The matrix is block diagonal,
because under isotropy the gradient components are independent from the
Hessian when evaluated at the same point. We can hence define a $5$-dimensional
standard Gaussian vector $Y(x)=(Y_{1}(x),Y_{2}(x),Y_{3}(x),Y_{4}(x),Y_{5}(x))$
with independent components such that
\begin{eqnarray*}
&&(\partial_{1}f_{\ell }(x),\partial_{2}f_{\ell }(x), \partial _{11}f_{\ell }(x),\partial _{12}f_{\ell }(x)-\cot \theta
_{x}\partial _{2}f_{\ell }(x),\partial _{22}f_{\ell }(x)+\cot \theta
_{x}\partial _{1}f_{\ell }(x)) \\
&&=\Lambda _{\ell }Y(x)\\
&&=\left( \tau _{1}Y_{1}(x),\tau _{1}Y_{2}(x),\tau
_{3}Y_{3}(x),\tau _{4}Y_{4}(x),\tau _{5}Y_{5}(x)+\tau _{2}Y_{3}(x)\right).
\end{eqnarray*}
Note that asymptotically
\begin{equation*}
\tau_{1}\sim \frac{\ell }{\sqrt{2}}, \hspace{0.7cm} \tau _{2}\sim \frac{\ell ^{2}}{
\sqrt{24}}, \hspace{0.7cm} \tau _{3}\sim \sqrt{\frac{3}{8}}\ell ^{2}, \hspace{0.7cm} \tau
_{4}\sim \frac{\ell ^{2}}{\sqrt{8}}, \hspace{0.7cm} \tau _{5}\sim \frac{\ell ^{2}}{\sqrt{3}},
\end{equation*}
where (as usual) $a_{\ell }\sim b_{\ell }$ means that the ratio between the
left- and right-hand side tends to unity as $\ell \rightarrow \infty$. Hence
\begin{align*}
Y_{a}(x) &=\frac{\sqrt{2}}{\sqrt{\lambda _{\ell }}}\partial _{a;x}f_{\ell
}(x), \hspace{1cm} a=1,2, \\
Y_{3}(x) &=\frac{2\sqrt{2}}{\sqrt{\lambda _{\ell }}\sqrt{3\lambda _{\ell }-2
}}\partial _{11;x}f_{\ell }(x), \\
Y_{4}(x) &=\frac{2\sqrt{2}}{\sqrt{\lambda _{\ell }}\sqrt{\lambda _{\ell }-2}
}\partial _{21;x}, \\
Y_{5}(x) &=\frac{\sqrt{3\lambda _{\ell }-2}}{\lambda _{\ell }\sqrt{{\lambda
_{\ell }-2}}}\partial _{22;x}f_{\ell }(x)-\frac{\lambda _{\ell }+2}{
\lambda _{\ell }\sqrt{{\lambda _{\ell }-2}}\sqrt{3\lambda _{\ell }-2}}
\partial _{11;x}f_{\ell }(x).
\end{align*}
Thus we obtain
\begin{align*}
\mathcal{N}_{\ell }^{c}& =\lim_{\varepsilon \to 0} \int_{\mathbb{S}^{2}}|\text{det}\nabla ^{2}f_{\ell
}(x)| {\boldsymbol \delta}_\varepsilon(\nabla f_{\ell }(x))dx \\
& =\lim_{\varepsilon \to 0} \int_{\mathbb{S}^{2}}|\partial _{11;x}f_{\ell }(x)\partial _{22;x}f_{\ell
}(x)-(\partial _{12;x}f_{\ell }(x))^{2}| {\boldsymbol \delta}_\varepsilon (\partial _{1;x}f_{\ell
}(x),\partial _{2;x}f_{\ell }(x))dx \\
& = \lim_{\varepsilon \to 0} \int_{\mathbb{S}^{2}}|\tau _{3}Y_{3}(x)(\tau _{5}Y_{5}(x)+\tau
_{2}Y_{3}(x))-(\tau _{4}Y_{4}(x))^{2}| {\boldsymbol \delta}_\varepsilon (\tau _{1}Y_{1}(x),\tau
_{1}Y_{2}(x))dx \\
& =\lim_{\varepsilon \to 0} \lambda_{\ell }^{2}  \int_{\mathbb{S}^{2}}\left\vert
\frac{\tau _{3}\tau _{5}}{\lambda _{\ell }^{2}}Y_{3}(x)Y_{5}(x)+\frac{\tau
_{2}\tau _{3}}{\lambda _{\ell }^{2}}Y_{3}^{2}(x)-\frac{\tau _{4}^{2}}{
\lambda _{\ell }^{2}}Y_{4}^{2}(x)\right\vert {\boldsymbol \delta}_\varepsilon (\tau_1 Y_{1}(x),\tau_1 Y_{2}(x))dx,
\end{align*}
where
\begin{equation*}
\frac{\tau _{3}\tau _{5}}{\lambda _{\ell }^{2}}\sim \frac{1}{\sqrt{8}}, \hspace{0.7cm}
\frac{\tau _{2}\tau _{3}}{\lambda _{\ell }^{2}}\sim \frac{1}{8}, \hspace{0.7cm}
\frac{\tau _{4}^{2}}{\lambda _{\ell }^{2}}\sim \frac{1}{8}.
\end{equation*}
The $q$th order chaos is the space generated by the $L^{2}$-completion of
linear combinations of the form $H_{q_{1}}(Y_{1})\cdots H_{q_{5}}(Y_{5}),$ with $
q_{1}+q_{2}+ \cdots +q_{5}=q$ (see i.e., \cite{noupebook})$;$ in other words, it
is the linear span of cross-product of Hermite polynomials computed in the
independent random variables $Y_{i},$ $i=1, 2,\dots 5$,  which generate the
gradient and Hessian of $f_{\ell }$. In particular, the $4$th order chaos
can be written in the following form:
\begin{align} \label{4chaos}
\mathcal{N}_{\ell }^{c}[4]&=\lambda_{\ell } \left[ \frac{1}{2!2!}
\sum_{i=2}^{5}\sum_{j=1}^{i-1}h_{ij} \int_{\mathbb{S}
^{2}}H_{2}(Y_{i}(x))H_{2}(Y_{j}(x))dx+\frac{1}{4!}\sum_{i=1}^{5}k_{i}
\int_{\mathbb{S}^{2}}H_{4}(Y_{i}(x))dx \right. \nonumber \\
&\;\; +\left. \frac{1}{3!} \sum_{\stackrel{i,j=1}{i \ne j}}^5 g_{ij} \int_{\mathbb{S}^{2}} H_3(Y_i(x)) H_1(Y_j(x)) d x  + \frac{1}{2} \sum_{\stackrel{i,j,k=1}{i \ne j \ne k}}^5 p_{ijk} \int_{\mathbb{S}^{2}} H_2(Y_i(x)) H_1(Y_j(x))H_1(Y_k(x)) d x   \right. \nonumber \\
&\;\; +\left. \sum_{\stackrel{i,j,k,l=1}{i \ne j \ne k \ne l}}^5 q_{ijkl} \int_{\mathbb{S}^{2}} H_1(Y_i(x)) H_1(Y_j(x)) H_1(Y_k(x))  H_1(Y_l(x)) d x
\right],
\end{align}
where
\begin{align*}
h_{ij}&=\lim_{\varepsilon \rightarrow 0} \lambda_{\ell } \; \mathbb{E}\left[ \left\vert
\frac{\tau _{3}\tau _{5}}{\lambda _{\ell }^{2}}Y_{3}Y_{5}+\frac{\tau
_{2}\tau _{3}}{\lambda _{\ell }^{2}}Y_{3}^{2}-\frac{\tau _{4}^{2}}{\lambda
_{\ell }^{2}}Y_{4}^{2}\right\vert {\boldsymbol \delta}_{\varepsilon
}(\tau_1 Y_{1}, \tau_1Y_{2})H_{2}(Y_{i})H_{2}(Y_{j})\right], \\
k_{i}&=\lim_{\varepsilon \rightarrow 0} \lambda_{\ell }\;  \mathbb{E}\left[ \left\vert
\frac{\tau _{3}\tau _{5}}{\lambda _{\ell }^{2}}Y_{3}Y_{5}+\frac{\tau
_{2}\tau _{3}}{\lambda _{\ell }^{2}}Y_{3}^{2}-\frac{\tau _{4}^{2}}{\lambda
_{\ell }^{2}}Y_{4}^{2}\right\vert {\boldsymbol \delta}_{\varepsilon
}(\tau_1 Y_{1}, \tau_1 Y_{2})H_{4}(Y_{i})\right],\\
g_{ij}&=\lim_{\varepsilon \rightarrow 0} \lambda_{\ell } \; \mathbb{E}\left[ \left\vert
\frac{\tau _{3}\tau _{5}}{\lambda _{\ell }^{2}}Y_{3}Y_{5}+\frac{\tau
_{2}\tau _{3}}{\lambda _{\ell }^{2}}Y_{3}^{2}-\frac{\tau _{4}^{2}}{\lambda
_{\ell }^{2}}Y_{4}^{2}\right\vert {\boldsymbol \delta}_{\varepsilon
}(\tau_1 Y_{1}, \tau_1Y_{2}) H_3(Y_i(x)) H_1(Y_j(x)) \right], \\
p_{ijk}&= \lim_{\varepsilon \rightarrow 0} \lambda_{\ell } \; \mathbb{E}\left[ \left\vert
\frac{\tau _{3}\tau _{5}}{\lambda _{\ell }^{2}}Y_{3}Y_{5}+\frac{\tau
_{2}\tau _{3}}{\lambda _{\ell }^{2}}Y_{3}^{2}-\frac{\tau _{4}^{2}}{\lambda
_{\ell }^{2}}Y_{4}^{2}\right\vert {\boldsymbol \delta}_{\varepsilon
}(\tau_1 Y_{1}, \tau_1Y_{2})H_2(Y_i(x)) H_1(Y_j(x))H_1(Y_k(x))  \right],\\
q_{ijkl}&=\lim_{\varepsilon \rightarrow 0} \lambda_{\ell } \; \mathbb{E}\left[ \left\vert
\frac{\tau _{3}\tau _{5}}{\lambda _{\ell }^{2}}Y_{3}Y_{5}+\frac{\tau
_{2}\tau _{3}}{\lambda _{\ell }^{2}}Y_{3}^{2}-\frac{\tau _{4}^{2}}{\lambda
_{\ell }^{2}}Y_{4}^{2}\right\vert {\boldsymbol \delta}_{\varepsilon
}(\tau_1 Y_{1}, \tau_1Y_{2}) H_1(Y_i(x)) H_1(Y_j(x)) H_1(Y_k(x))  H_1(Y_l(x))\right].
\end{align*}
\noindent The projection coefficients $k_{i}$, $h_{ij}$, $g_{ij}$, $p_{ijk}$, and $q_{ijkl}$ are constant with respect to $\ell$.

\section{Proof of Theorem \protect\ref{tisone1} \label{DimMain}}

In this section we give the proof of our main result. Let us start with the $L^{2}(\Omega )$, $\varepsilon$-approximation to
the number of critical points \cite{CM2018b}
\begin{equation*}
 \mathcal{N}_{\ell }^{c}=\lim_{\varepsilon \rightarrow 0}\mathcal{N}_{\ell,\varepsilon }^{c}, \hspace{1cm}
 \mathcal{N}_{\ell ,\varepsilon }^{c}=\int_{\mathbb{S}^{2}}|\text{det}\nabla
^{2}f_{\ell }(x)| {\boldsymbol \delta}_{\varepsilon }(\nabla f_{\ell }(x))dx,
\end{equation*}
for every $x\in \mathbb{S}^{2}$ we define
\begin{equation*}
|\text{det}\nabla ^{2}f_{\ell }(x)| {\boldsymbol \delta}_{\varepsilon }(\nabla f_{\ell
}(x))=\sum_{q=0}^{\infty }\psi _{\ell }^{\varepsilon }(x;q)=:\psi _{\ell
}^{\varepsilon }(x).
\end{equation*}
By continuity of the inner product in $L^{2}(\Omega )$, we write
\begin{align*}
\text{Cov}(\mathcal{N}_{\ell }^{c},h_{\ell ;4}) &=\lim_{\varepsilon \rightarrow
0}\text{Cov}(\mathcal{N}_{\ell ,\varepsilon }^{c},h_{\ell ;4})  \\
&=\lim_{\varepsilon \rightarrow 0}\mathbb{E}\left[ \int_{\mathbb{S}^{2}}|
\text{det}\nabla ^{2}f_{\ell }(x)| {\boldsymbol \delta}_{\varepsilon }(\nabla f_{\ell
}(x))dx\int_{\mathbb{S}^{2}}H_{4}(f_{\ell }(y))dy\right] \\
&=\lim_{\varepsilon \rightarrow 0}\mathbb{E}\left[ \int_{\mathbb{S}
^{2}}\sum_{q=0}^{\infty }\psi _{\ell }^{\varepsilon }(x;q)dx\int_{\mathbb{S}
^{2}}H_{4}(f_{\ell }(y))dy\right].
\end{align*}
Now note that both $\psi _{\ell }^{\varepsilon }(x)$ and $H_{4}(f_{\ell
}(y)) $ are isotropic processes on $\mathbb{S}^{2}$, and hence we have
\begin{align*}
\mathbb{E}\left[ \int_{\mathbb{S}^{2}}\sum_{q=0}^{\infty }\psi _{\ell
}^{\varepsilon }(x;q)dx\int_{\mathbb{S}^{2}}H_{4}(f_{\ell }(y))dy\right]
&=\mathbb{E}\left[ \int_{\mathbb{S}^{2}}\lim_{Q\rightarrow \infty
}\sum_{q=0}^{Q}\psi _{\ell }^{\varepsilon }(x;q)dx\int_{\mathbb{S}
^{2}}H_{4}(f_{\ell }(y))dy\right] \\
&=\lim_{Q\rightarrow \infty }\mathbb{E}\left[ \int_{\mathbb{S}
^{2}}\sum_{q=0}^{Q}\psi _{\ell }^{\varepsilon }(x;q)dx\int_{\mathbb{S}
^{2}}H_{4}(f_{\ell }(y))dy\right]
\end{align*}
by continuity of covariances. Moreover because all integrands are finite-order polynomials we have
\begin{align*}
\lim_{Q\rightarrow \infty }\mathbb{E}\left[ \int_{\mathbb{S}
^{2}}\sum_{q=0}^{Q}\psi _{\ell }^{\varepsilon }(x;q)dx\int_{\mathbb{S}
^{2}}H_{4}(f_{\ell }(y))dy\right]
&=\lim_{Q\rightarrow \infty }\sum_{q=0}^{Q}\int_{\mathbb{S}^{2}}\int_{
\mathbb{S}^{2}}\mathbb{E}\left[ \psi _{\ell }^{\varepsilon
}(x;q)H_{4}(f_{\ell }(y))\right] dxdy  \\
&=\int_{\mathbb{S}^{2}}\int_{\mathbb{S}^{2}}\mathbb{E}\left[ \psi _{\ell
}^{\varepsilon }(x;4)H_{4}(f_{\ell }(y))\right] dxdy \\
&=16\pi ^{2}\int_{0}^{\pi /2}\mathbb{E}\left[ \psi _{\ell }^{\varepsilon }(
\overline{x};4)H_{4}(f_{\ell }(y(\phi )))\right] \sin \phi d\phi,
\end{align*}
where in the last steps we used orthogonality of Wiener chaoses and isotropy;
we take $\overline{x}=(\frac{\pi }{2},0)$ and $y(\phi )=(\frac{\pi }{2},\phi )$.
More explicitly, the previous argument allows us to perform our
argument on {\it the equator}, where $\theta $ is fixed to $\pi/2$.
Note that
\begin{align*}
\psi _{\ell }^{\varepsilon }(\overline{x};4)&=\lambda_{\ell} \Big[\frac{1}{2!2!}\sum_{i=2}^{5}\sum_{j=1}^{i-1}h_{ij}^{
\varepsilon }H_{2}(Y_{i}(\bar{x}))H_{2}(Y_{j}(\bar{x}))+\frac{1}{4!}
\sum_{i=1}^{5}k_{i}^{\varepsilon }H_{4}(Y_{i}(\bar{x}))\\
&\;\;+ \frac{1}{3!} \sum_{\stackrel{i,j=1}{i \ne j}}^5 g_{ij}^{\varepsilon }  H_3(Y_i(x)) H_1(Y_j(x))  + \frac{1}{2} \sum_{\stackrel{i,j,k=1}{i \ne j \ne k}}^5 q_{ijk}^{\varepsilon }  H_2(Y_i(x)) H_1(Y_j(x))H_1(Y_k(x))   \\
&\;\; + \sum_{\stackrel{i,j,k,l=1}{i \ne j \ne k \ne l}}^5 p^{\varepsilon }_{ijkl}  H_1(Y_i(x)) H_1(Y_j(x)) H_1(Y_k(x))  H_1(Y_l(x)) \Big],
\end{align*}
and hence
\begin{align*}
&\text{Cov}(\mathcal{N}_{\ell }^{c},h_{\ell ;4}) \\
&=16\pi ^{2}\lim_{\varepsilon
\rightarrow 0}\int_{0}^{\pi /2}\mathbb{E}\left[ \psi _{\ell }^{\varepsilon }(
\overline{x};4)H_{4}(f_{\ell }(y(\phi )))\right] \sin \phi d\phi \\
&=16\pi^{2} \lambda_{\ell}  \frac{1}{2! 2!}
\sum_{i=2}^{5}\sum_{j=1}^{i-1} \{\lim_{\varepsilon \rightarrow
0}h_{ij}^{\varepsilon } \} \int_{0}^{\pi /2}\mathbb{E}\left[ H_{2}(Y_{i}(
\bar{x}))H_{2}(Y_{j}(\bar{x}))H_{4}(f_{\ell }(y(\phi )))\right] \sin \phi
d\phi \\
&\;\;+16\pi^{2} \lambda_{\ell} \frac{1}{4!}
\sum_{i=1}^{5} \{\lim_{\varepsilon \rightarrow 0}k_{i}^{\varepsilon }
\} \int_{0}^{\pi /2}\mathbb{E}\left[ H_{4}(Y_{i}(\bar{x}))H_{4}(f_{\ell
}(y(\phi )))\right] \sin \phi d\phi\\
&\;\;+16\pi^{2} \lambda_{\ell} \frac{1}{3!}
\sum_{i \ne j} \{\lim_{\varepsilon \rightarrow 0}g_{ij}^{\varepsilon }
\} \int_{0}^{\pi /2} \mathbb{E}\left[ H_{3}(Y_{i}(\bar{x})) H_{1}(Y_{j}(\bar{x})) H_{4}(f_{\ell
}(y(\phi )))\right] \sin \phi d\phi\\
&\;\;+16\pi^{2} \lambda_{\ell} \frac{1}{2}
\sum_{i \ne j \ne k} \{\lim_{\varepsilon \rightarrow 0} p_{ijk}^{\varepsilon }
\} \int_{0}^{\pi /2}\mathbb{E}\left[ H_{2}(Y_{i}(\bar{x})) H_{1}(Y_{j}(\bar{x})) H_{1}(Y_{k}(\bar{x})) H_{4}(f_{\ell
}(y(\phi )))\right] \sin \phi d\phi\\
&\;\;+16\pi^{2} \lambda_{\ell}
\sum_{i \ne j \ne k \ne l} \{\lim_{\varepsilon \rightarrow 0} q_{ijkl}^{\varepsilon }
\} \int_{0}^{\pi /2}\mathbb{E}\left[ H_{1}(Y_{i}(\bar{x})) H_{1}(Y_{j}(\bar{x})) H_{1}(Y_{k}(\bar{x})) H_{1}(Y_{l}(\bar{x})) H_{4}(f_{\ell
}(y(\phi )))\right] \sin \phi d\phi.
\end{align*}
We shall show below that the asymptotic behaviour of $\text{Cov}(\mathcal{N}_{\ell
}^{c},h_{\ell ;4})$ is dominated by three terms corresponding to
\begin{equation*}
\int_{0}^{\pi /2}\mathbb{E}\left[ H_{4}(Y_{2}(\bar{x}))H_{4}(f_{\ell
}(y(\phi )))\right] \sin \phi d\phi, \;\;\; \int_{0}^{\pi /2}\mathbb{E}
\left[ H_{4}(Y_{5}(\bar{x}))H_{4}(f_{\ell }(y(\phi )))\right] \sin \phi d\phi,
\end{equation*}
and
\begin{equation*}
\int_{0}^{\pi /2}\mathbb{E}\left[ H_{2}(Y_{2}(\bar{x}))H_{2}(Y_{5}(\bar{x}
))H_{4}(f_{\ell }(y(\phi )))\right] \sin \phi d\phi.
\end{equation*}
The computation of these leading covariances is given in the three Lemmas \ref
{domin1}-\ref{domin3} to follow, where it is shown that
\begin{equation*}
\int_{0}^{\pi /2}\mathbb{E}\left[ H_{4}(Y_{2}(\bar{x}))H_{4}(f_{\ell
}(y(\phi )))\right] \sin \phi d\phi =4!\frac{2 \cdot 3}{\pi ^{2}}\frac{\log \ell }{\ell ^{2}}+O(\ell ^{-2}),
\end{equation*}
\begin{equation*}
\int_{0}^{\pi /2}\mathbb{E}\left[ H_{4}(Y_{5}(\bar{x}))H_{4}(f_{\ell
}(y(\phi )))\right] \sin \phi d\phi =4!\frac{3^3}{2\pi ^{2}}\frac{\log
\ell }{\ell ^{2}}+O(\ell ^{-2}),
\end{equation*}
\begin{equation*}
\int_{0}^{\pi /2}\mathbb{E}\left[ H_{2}(Y_{2}(\bar{x}))H_{2}(Y_{5}(\bar{x}
))H_{4}(f_{\ell }(y(\phi )))\right] \sin \phi d\phi =4!\frac{3}{\pi ^{2}}
\frac{\log \ell }{\ell ^{2}}+O(\ell ^{-2}).
\end{equation*}
All the remaining terms in $\text{Cov}(\mathcal{N}_{\ell }^{c},h_{\ell;4})$
are shown to be $O(\ell^{-2})$ or smaller in Section \ref{oddt} and Lemmas \ref{subdomin1}-\ref{odd3} below.
From Proposition \ref{coefficients} we know that
\begin{equation*}
k_{2}=\lim_{\varepsilon \rightarrow 0 }k^{\varepsilon}_{2}= \frac{1}{\pi }\frac{\sqrt{3}}{2}, \hspace{1cm} k_{5}=\lim_{\varepsilon \rightarrow 0 }k^{\varepsilon}_{5}=-\frac{1}{\pi }
\frac{7}{ 3^{3}\sqrt{3}}, \hspace{1cm} h_{25}:=\lim_{\varepsilon \rightarrow 0}h^{\varepsilon}_{25}=-\frac{1}{\pi }
\frac{1}{3\sqrt{3}}.
\end{equation*}
Substituting and after some straightforward algebra, one obtains
\begin{align*}
\text{Cov}(\mathcal{N}_{\ell }^{c},h_{\ell ;4}) &=\lambda _{\ell }\left\{ 4 \pi
^{2}h_{25}4!\frac{3}{\pi ^{2}}\frac{\log \ell }{\ell ^{2}}+\frac{2}{3}\pi
^{2}k_{2}4!2^{2}\frac{3}{\pi ^{2}}\frac{\log \ell }{2\ell ^{2}}+\frac{2}{3}
\pi ^{2}k_{5}4!3^{2}\frac{3}{\pi ^{2}}\frac{\log \ell }{2\ell ^{2}}+O(\ell
^{-2})\right\} \\
&=\frac{\lambda _{\ell }}{3\sqrt{3}}4!\frac{\log \ell }{\ell ^{2}}\frac{1}{
\pi }\times \left\{ -12+18-7\right\} +O(1) \\
&=-\frac{\lambda _{\ell }}{3\sqrt{3}}4!\frac{\log \ell }{\ell ^{2}}\frac{1}{
\pi }+O(1).
\end{align*}
Because
\begin{equation*}
\mathcal{A}_{\ell }=-\frac{\lambda _{\ell }}{2^{3}3^{2}\sqrt{3}\pi }\int_{
\mathbb{S}^{2}}H_{4}(f_{\ell }(x))dx=-\frac{\lambda _{\ell }}{2^{3}3^{2}
\sqrt{3}\pi }h_{\ell ;4},
\end{equation*}
we find
\begin{align*}
\text{Cov}(\mathcal{N}_{\ell }^{c},\mathcal{A}_{\ell }) =\frac{\lambda _{\ell
}^{2}}{2^{3}3^{2}\sqrt{3}\pi }\frac{1}{3\sqrt{3}}4!\frac{\log \ell }{\ell
^{2}}\frac{1}{\pi }+O(1) =\frac{\log \ell }{3^{3}\pi ^{2}}+O(1),
\end{align*}
so that our proof of our main theorem is completed, recalling that, as $\ell \to \infty$,
\begin{equation*}
\text{Var}(\mathcal{N}_{\ell }^{c})\sim \text{Var}(\mathcal{A}_{\ell })=\frac{\log \ell }{
3^{3}\pi ^{2}}+O(1).
\end{equation*}

\begin{remark}
A consequence of Theorem \ref{tisone1} is that, as $\ell \rightarrow \infty$,
\begin{equation*}
\text{Var}\left( \text{Proj}[\mathcal{N}_{\ell }^{c}|4]\right)=\frac{\ell ^{2}\log \ell }{
3^{3}\pi ^{2}}+O(\ell ^{2}),
\end{equation*}
so that
\begin{equation*}
\lim_{\ell \rightarrow \infty }\frac{\text{Var}\left( \text{Proj}[\mathcal{N}_{\ell
}^{c}|4]\right) }{\text{Var}\left( \mathcal{N}_{\ell }^{c}\right) }=1.
\end{equation*}
Note that by orthogonality we have
\begin{equation*}
\text{Var}\left( \mathcal{N}_{\ell }^{c}\right) =\sum_{q=0}^{\infty }\text{Var}\left(
\text{Proj}[\mathcal{N}_{\ell }^{c}|q]\right) =\text{Var}\left( \text{Proj}[
\mathcal{N}_{\ell }^{c}|4]\right) +\sum_{k=1}^{\infty }\text{Proj}[\mathcal{
N}_{\ell }^{c}|4+2k],
\end{equation*}
where the odd terms in the expansion vanish by symmetry arguments,
 $\rm{Var}\left( \text{Proj}[\mathcal{N}_{\ell }^{c}|0]\right) =0$ is obvious and
$\rm{Var}\left( \text{Proj}[\mathcal{N}_{\ell }^{c}|2]\right) =0$ was shown in
\cite{CM2018b}. Hence we have the bound
\begin{equation*}
\sum_{k=1}^{\infty }\text{Proj}[\mathcal{N}_{\ell }^{c}|4+2k]=o(\ell
^{2}\log \ell ).
\end{equation*}
In fact, it is possible to establish the slightly stronger result
\begin{equation*}
\sum_{k=1}^{\infty }\text{Proj}[\mathcal{N}_{\ell }^{c}|4+2k]=O(\ell^{2});
\end{equation*}
we omit the proof for brevity's sake.
\end{remark}

\section{Evaluation of the Projection Coefficients $h_{52},k_{2},k_{5}$
\label{EvaCoeff}}

In this section we evaluate the three projection coefficients in the
Wiener-chaos expansion which are required for the completion of our
arguments.
\begin{proposition}
\label{coefficients} We have that
\begin{align*}
k_{2}=\frac{1}{\pi }\frac{
\sqrt{3}}{2},\hspace{1cm}
k_{5}=-\frac{1}{\pi }\frac{7
}{ 3^{3}\sqrt{3}},\hspace{1cm}
h_{25}=-\frac{1}{\pi }\frac{1}{3\sqrt{3}}.
\end{align*}
\end{proposition}
\begin{proof}
Let us recall first the following simple result
\begin{equation*}
\varphi_{a}:=\lim_{\varepsilon \to 0} \mathbb{E}[H_{a}(Y)\delta_{\varepsilon}(\tau_1 Y)]=
\begin{cases}
\frac{1}{\sqrt{2\pi }\tau_1} & a=0, \\
0 & a=1, \\
-\frac{1}{\sqrt{2\pi }\tau_1} & a=2, \\
\frac{3}{\sqrt{2\pi }\tau_1} & a=4.
\end{cases}
\end{equation*}
Indeed, for example
\begin{equation*}
\lim_{\varepsilon \to 0} \mathbb{E}[H_{4}(Y)\delta_{\varepsilon} (\tau_1 Y)]=\lim_{\varepsilon \to 0} \mathbb{E}[(Y^{4}-6Y^{2}+3) \delta_{\varepsilon} (\tau_1 Y)]=\frac{3}{\sqrt{2\pi }\tau_1},
\end{equation*}
since
\begin{equation*}
\lim_{\varepsilon \to 0} \mathbb{E}[Y^{n} \delta_{\varepsilon}(\tau_1 Y)]=\lim_{\varepsilon \to 0}\int_{-\infty }^{\infty }y^{n}\delta_{\varepsilon} (\tau_1 y)\frac{1}{
\sqrt{2\pi }}e^{-\frac{y^{2}}{2}}dy=
\begin{cases}
\frac{1}{\sqrt{2\pi }\tau_1} & n=0, \\
0 & n=1,2,3\dots
\end{cases}
\end{equation*}
Now note that
\begin{align*}
k_{2}&=\lim_{\varepsilon \rightarrow 0 }k^{\varepsilon}_{2} \\
&=  \lambda_{\ell } \; \mathbb{E}\left[ \left\vert \frac{1}{2\sqrt{2}}Y_{3}Y_{5}+\frac{1}{8}
Y_{3}^{2}-\frac{1}{8}Y_{4}^{2}\right\vert \right] \varphi _{0}\,\varphi _{4} \\
&= \frac{3}{\pi }  \mathbb{E}\left[ \left\vert \frac{1}{2\sqrt{2}}Y_{3}Y_{5}+
\frac{1}{8}Y_{3}^{2}-\frac{1}{8}Y_{4}^{2}\right\vert \right],
\end{align*}
\begin{align*}
k_{5}&=\lim_{\varepsilon \rightarrow 0 }k^{\varepsilon}_{5} \\
&= \lambda_{\ell } \;\mathbb{E}\left[ \left\vert \frac{1}{2\sqrt{2}}
Y_{3}Y_{5}+\frac{1}{8}Y_{3}^{2}-\frac{1}{8}Y_{4}^{2}\right\vert H_{4}(Y_{5})
\right] \varphi _{0}^{2} \\ &= \frac{1}{\pi }\;\mathbb{E}\left[ \left\vert \frac{1}{2\sqrt{2}}
Y_{3}Y_{5}+\frac{1}{8}Y_{3}^{2}-\frac{1}{8}Y_{4}^{2}\right\vert H_{4}(Y_{5})
\right],
\end{align*}
and
\begin{align*}
h_{52}&=\lim_{\varepsilon \rightarrow 0}h^{\varepsilon}_{25} \\
&=  \lambda_{\ell }\; \mathbb{E}\left[ \left\vert \frac{1}{2\sqrt{2}}Y_{3}Y_{5}+\frac{1}{8}
Y_{3}^{2}-\frac{1}{8}Y_{4}^{2}\right\vert H_{2}(Y_{5})\right] \varphi
_{0}\,\varphi _{2} \\
 &=  - \frac{1}{\pi } \mathbb{E}\left[ \left\vert \frac{1}{2\sqrt{2}}Y_{3}Y_{5}+
\frac{1}{8}Y_{3}^{2}-\frac{1}{8}Y_{4}^{2}\right\vert H_{2}(Y_{5})\right].
\end{align*}
Let us introduce the change of variables
\begin{equation*}
Z_{1}=\sqrt{3}Y_{3},\hspace{1cm}Z_{2}=Y_{4},\hspace{1cm}Z_{3}=\frac{\sqrt{8}
}{\sqrt{3}}Y_{5}+\frac{1}{\sqrt{3}}Y_{3},
\end{equation*}
so that $(Z_1, Z_2, Z_2)$ is a centred Gaussian vector with covariance matrix
\begin{equation*}
\left(
\begin{matrix}
3 & 0 & 1 \\
0 & 1 & 0 \\
1 & 0 & 3
\end{matrix}
\right),
\end{equation*}
and we can write
\begin{align*}
\frac{1}{2\sqrt{2}}Y_{3}Y_{5}+\frac{1}{8}Y_{3}^{2}-\frac{1}{8}Y_{4}^{2}
=\frac{1}{8} (Z_{1}Z_{3}- Z_{2}^{2}).
\end{align*}
\noindent The coefficient $k_{2}$ can be computed as follows: write
\begin{equation*}
\mathbb{E}\left[ \left\vert \frac{1}{\sqrt{8}}Y_{3}Y_{5}+\frac{1}{8}
Y_{3}^{2}-\frac{1}{8}Y_{4}^{2}\right\vert \right] =\frac{1}{8}\mathbb{E}
\left[ \left\vert Z_{1}Z_{3}-Z_{2}^{2}\right\vert \right]=\frac{1}{8} \mathbb{
E}\left[ \left\vert (Z_{1},Z_{2},Z_{3})^{T}A(Z_{1},Z_{2},Z_{3})\right\vert
\right],
\end{equation*}
where the symmetric matrix $A$ given by
\begin{equation*}
A=\left(
\begin{matrix}
0 & 0 & {1}/{2} \\
0 & -1 & 0 \\
{1}/{2} & 0 & 0
\end{matrix}
\right);
\end{equation*}
we apply \cite{LiWei}, Theorem 2.1, to obtain
\begin{equation*}
\mathbb{E}\left[ \left\vert Z_{1}Z_{3}-Z_{2}^{2}\right\vert \right] =\frac{2
}{\pi }\int_{0}^{\infty }\frac{1}{t^{2}}\left\{ 1-\frac{1}{2\sqrt{\text{det}
(I-2it\Sigma A)}}-\frac{1}{2\sqrt{\text{det}(I+2it\Sigma A)}}\,\right\} dt
\end{equation*}
where we have that
\begin{equation*}
\text{det}(I-2it\Sigma A)=1+12t^{2}+16it^{3}, \hspace{1cm} \text{ det}(I+2it\Sigma
A)=1+12t^{2}-16it^{3},
\end{equation*}
and computing the integral with Cauchy methods for residuals, we get
\begin{equation*}
\mathbb{E}\left[ \left\vert Z_{1}Z_{3}-Z_{2}^{2}\right\vert \right] =\frac{
4}{\sqrt{3}},
\end{equation*}
and
\begin{equation*}
k_{2} =\frac{3}{\pi }\mathbb{E}\left[ \left\vert \frac{1}{\sqrt{8}}
Y_{3}Y_{5}+\frac{1}{8}Y_{3}^{2}-\frac{1}{8}Y_{4}^{2}\right\vert \right] =
\frac{\sqrt{3}}{2 \pi },
\end{equation*}
as claimed.
We introduce now the following notation
\begin{equation*}
\mathcal{I}_{r}=\mathbb{E}[|Z_{1}Z_{3}-Z_{2}^{2}|(Z_{1}-3Z_{3})^{r}]
\end{equation*}
for $r=0,2,4$, so that,
\begin{align*}
h_{52}
& =-\frac{1}{\pi }\frac{1}{8}\mathbb{E}\left[ \left\vert
Z_{1}Z_{3}-Z_{2}^{2}\right\vert H_{2}\left( \frac{1}{\sqrt{8}\sqrt{3}}
(3Z_{3}-Z_{1})\right) \right]\\
& =-\frac{1}{\pi }\frac{1}{8}\mathbb{E}\left[ \left\vert
Z_{1}Z_{3}-Z_{2}^{2}\right\vert \left( \frac{1}{\sqrt{8}\sqrt{3}}
(3Z_{3}-Z_{1})\right) ^{2}\right] +\frac{1}{\pi }\frac{1}{8}\mathbb{E}\left[
\left\vert Z_{1}Z_{3}-Z_{2}^{2}\right\vert \right] \\
& =-\frac{1}{\pi }\frac{1}{3\cdot 2^{6}}\mathcal{I}_{2}+\frac{}{\pi }\frac{1
}{2^{3}}\mathcal{I}_{0},
\end{align*}
and
\begin{align*}
k_{5}
& =\frac{1}{\pi }\;\frac{1}{8}\mathbb{E}\left[ \left\vert
Z_{1}Z_{3}-Z_{2}^{2}\right\vert H_{4}\left( \frac{1}{\sqrt{8}\sqrt{3}}
(3Z_{3}-Z_{1})\right) \right] 
 =\frac{1}{\pi }\frac{1}{2^{9}\cdot 3^{2}}\mathcal{I}_{4}-\frac{1}{\pi }\;
\frac{1}{2^{5}}\mathcal{I}_{2}+\frac{1}{\pi }\;\frac{3}{2^{3}}\mathcal{I}_{0}.
\end{align*}
The statement follows applying the results in \cite{CW} where it is proved that
\begin{equation*}
\mathcal{I}_{2}=\frac{
2^{5}\cdot 5}{\sqrt{3}},\hspace{1cm}\mathcal{I}_{4}=\frac{2^{8}\cdot
5^{2}\cdot 7}{3\sqrt{3}}.
\end{equation*}

\end{proof}

\section{Terms with odd index Hermite polynomials} \label{oddt}

In this section we prove that the terms in the $4$-th chaos formula \eqref{4chaos} with odd index Hermite polynomials produce in $\text{Cov}(\mathcal{N}_{\ell }^{c},h_{\ell;4})$ terms of order $O(\ell^{-2})$ and terms equal to zero. We first focus, in the following proposition, on the projection coefficients.
\begin{proposition} \label{17:37}
The projection coefficients $g_{ij}$, $p_{ijk}$ and $q_{ijkl}$ are such that
\begin{itemize}
\item for $i, j \ne 3, 5$, we have  $g_{ij}=0$,
\item for $j, k \ne 1, 2, 4$, we have $p_{ijk}=0$,
\item we have $q_{ijkl}=0$.
\end{itemize}
\end{proposition}
\begin{proof}
Recalling that for $a$ odd we have
\begin{equation*}
\lim_{\varepsilon \to 0} \mathbb{E}[H_{a}(Y)\delta_{\varepsilon}(\tau_1 Y)]=0,
\end{equation*}
from this we immediately see that the coefficients $g_{ij}$ with $i, j = 1, 2$ are all equal to zero. We consider now the coefficients $g_{ij}$ with $i=4$ or $j=4$, for these coefficients we observe that the expectation with respect to the random variable $Y_4$ vanishes since it is expressed as the integral of an odd function. The proof of the last two points of the statement is similar.
\end{proof}

In Lemmas \ref{odd1}-\ref{odd3} we prove that the terms in $\text{Cov}(\mathcal{N}_{\ell }^{c},h_{\ell;4})$ that are multiplied by the projection coefficients not discussed in Proposition \ref{17:37} are either zero or of order $O(\ell^{-2})$. In particular we prove that, for $a, b=3, 5$, $a \ne b$,  
\begin{align*}
\int_{0}^{\pi / 2} \mathbb{E}[H_{3}(Y_{a}(\bar{x}))H_{1}(Y_{b}(\bar{x})) H_4(f_{\ell}(y(\phi)))] \sin \phi d \phi= O(\ell^{-2}),
\end{align*}
for $a=1, 2$,
\begin{align*}
\int_{0}^{\pi / 2} \mathbb{E}[H_{2}(Y_{a}(\bar{x})) H_{1}(Y_{3}(\bar{x})) H_{1}(Y_{5}(\bar{x}))  H_4(f_{\ell}(y(\phi)))] \sin \phi d \phi=0,
\end{align*}
and that
\begin{align*}
 \int_{0}^{\pi / 2} \mathbb{E}[H_{2}(Y_{4}(\bar{x})) H_{1}(Y_{3}(\bar{x}))  H_{1}(Y_{5}(\bar{x})) H_4(f_{\ell}(y(\phi)))] \sin \phi d \phi=0.
 \end{align*}

\appendix

\section{Auxiliary Lemmas}\label{Appendix}

In this appendix, we collect a number of technical results that were
exploited for the correlation results above. We divide the results into two
subsections, collecting respectively dominant and subdominant terms.

\subsection{Dominant terms}

In this subsection, we collect the results concerning the three dominant
terms.
\begin{lemma}
\label{domin1} \bigskip As $\ell \rightarrow \infty$,
\begin{equation*}
\int_{0}^{\pi /2}\mathbb{E}\left[ H_{4}(Y_{2}(\bar{x}))H_{4}(f_{\ell
}(y(\phi )))\right] \sin \phi d\phi =4!2^{2}\frac{3}{\pi ^{2}}\frac{\log
\ell }{2\ell ^{2}}+O(\ell ^{-2}).
\end{equation*}
\end{lemma}
\begin{proof}
Note first that, by Diagram Formula (see \cite{MaPeCUP} Section 4.3.1),
\begin{align*}
\mathbb{E}\left[ H_{4}(Y_{2}(\bar{x}))H_{4}(f_{\ell }(y(\phi )))\right]
&=4!\left\{ \mathbb{E}\left[ Y_{2}(\bar{x}))f_{\ell }(y(\phi ))\right]
\right\} ^{4} \\
&=4!\left\{ \mathbb{E}\left[ \sqrt{\frac{2}{\ell (\ell +1)}}\partial
_{2;x}f_{\ell }(x)f_{\ell }(y(\phi ))\right] \right\} ^{4} \\
&=4!\frac{2^{2}}{\ell ^{2}(\ell +1)^{2}}\left\{ \mathbb{E}\left[ \partial
_{2;x}f_{\ell }(x)f_{\ell }(y(\phi ))\right] \right\} ^{4}.
\end{align*}
Now we have easily
\begin{equation*}
\left\langle x,y\right\rangle =\cos \theta _{x}\cos \theta _{y}+\sin \theta
_{x}\sin \theta _{y}\cos (\varphi _{x}-\varphi _{y}),
\end{equation*}
and
\begin{align*}
\mathbb{E}\left[ \partial _{2;x}f_{\ell }(x)f_{\ell }(y(\phi ))\right]
&=\left. \frac{1}{\sin \theta }\frac{\partial }{\partial \varphi _{x}}
P_{\ell }(\left\langle x,y\right\rangle )\right\vert _{x=\overline{x},y=y(\phi )} \\
&=-\left. \frac{1}{\sin \theta }P_{\ell }^{\prime }(\left\langle
x,y\right\rangle )\sin \theta _{x}\sin \theta _{y}\sin (\varphi _{x}-\varphi
_{y})\right\vert _{x=\overline{x},y=y(\phi )} \\
&=P_{\ell }^{\prime }(\cos \phi )\sin (\phi ).
\end{align*}
Thus we obtain
\begin{align*}
\int_{0}^{\pi /2}\mathbb{E}\left[ H_{4}(Y_{2}(\bar{x}))H_{4}(f_{\ell
}(y(\phi )))\right] \sin \phi d\phi &=4!\frac{2^{2}}{\ell ^{2}(\ell +1)^{2}}
\int_{0}^{\pi /2}\left\{ P_{\ell }^{\prime }(\cos \phi )\sin (\phi )\right\}
^{4}\sin \phi d\phi \\
&=4!\frac{2^{2}}{\ell ^{2}(\ell +1)^{2}}\frac{3\ell ^{4}}{\pi ^{2}}\frac{
\log \ell }{\ell ^{2}}+O(\ell ^{-2}) \\
&=4!\frac{12}{\pi ^{2}}\frac{\log \ell }{\ell ^{2}}+O(\ell ^{-2}),
\end{align*}
using, see Lemma \ref{AA} below,
\begin{equation*}
\int_{0}^{\pi /2}\;[P_{\ell }^{(r)}(\cos \phi )\sin ^{r}\phi ]^{4}\sin \phi
d\phi =\frac{3\ell ^{4r}}{\pi ^{2}}\frac{\log \ell }{2\ell ^{2}}+O(\ell
^{4r-2}).
\end{equation*}
\end{proof}

\begin{lemma}
\label{domin2} As $\ell \rightarrow \infty $
\begin{equation*}
\int_{0}^{\pi /2}\mathbb{E}\left[ H_{4}(Y_{5}(\bar{x}))H_{4}(f_{\ell
}(y(\phi )))\right] \sin \phi d\phi =4!3^{2}\frac{3}{\pi ^{2}}\frac{\log
\ell }{2\ell ^{2}}+O(\ell ^{-2}).
\end{equation*}
\end{lemma}
\begin{proof}
As before, note first that
\begin{align*}
&\mathbb{E}\left[ H_{4}(Y_{5}(\bar{x}))H_{4}(f_{\ell }(y(\phi )))\right]\\
&=4!\left\{ \mathbb{E}\left[ Y_{5}(\bar{x})f_{\ell }(y(\phi ))\right]
\right\} ^{4}\\
&=4!\left\{ \mathbb{E}\left[ \left( \frac{\sqrt{3\lambda _{\ell }-2}}{
\lambda _{\ell }\sqrt{{\lambda _{\ell }-2}}}\partial _{22;x}f_{\ell }(x)-
\frac{(\lambda _{\ell }+2)}{\lambda _{\ell }\sqrt{{\lambda _{\ell }-2}}\sqrt{
3\lambda _{\ell }-2}}\partial _{11;x}f_{\ell }(x)\right) f_{\ell }(y(\phi ))
\right] \right\} ^{4} \\
&=4!\left\{ \mathbb{E}\left[ \left( \alpha _{1\ell }\partial _{22;x}f_{\ell
}(x)-\alpha _{2\ell }\partial _{11;x}f_{\ell }(x)\right) f_{\ell }(y(\phi ))
\right] \right\} ^{4},
\end{align*}
where we wrote
\begin{equation*}
\alpha _{1\ell }:=\frac{\sqrt{3\lambda _{\ell }-2}}{\lambda _{\ell }\sqrt{{
\lambda _{\ell }-2}}}, \;\;\;\; \alpha _{2\ell }:=\frac{(\lambda _{\ell }+2)}{
\lambda _{\ell }\sqrt{{\lambda _{\ell }-2}}\sqrt{3\lambda _{\ell }-2}};
\end{equation*}
note that
\begin{equation*}
\alpha _{1\ell }=\frac{\sqrt{3}}{\ell ^{2}}+O(\frac{1}{\ell ^{3}}), \;\;\;\;
\alpha _{2\ell }=\frac{1}{\sqrt{3}\ell ^{2}}+O(\frac{1}{\ell ^{3}}).
\end{equation*}
Now
\begin{align*}
\mathbb{E}\left[ \partial _{22;x}f_{\ell }(x)f_{\ell }(y(\phi ))\right]
&=\left. \frac{1}{\sin ^{2}\theta _{x}}\frac{\partial ^{2}}{\partial
\varphi _{x}^{2}}P_{\ell }(\left\langle x,y\right\rangle )\right\vert _{x=\overline{x},y=y(\phi )}
=P_{\ell }^{\prime \prime }(\cos \phi )\sin ^{2}\phi -P_{\ell }^{\prime
}(\cos \phi )\cos \phi.
\end{align*}
Likewise
\begin{align*}
\mathbb{E}\left[ \partial _{11;x}f_{\ell }(x)f_{\ell }(y(\phi )\right]
&=\left. \frac{\partial ^{2}}{\partial \theta _{x}^{2}}P_{\ell
}(\left\langle x,y\right\rangle )\right\vert _{x=\overline{x},y=y(\phi )}
=-P_{\ell }^{\prime }(\cos \phi )\cos \phi.
\end{align*}
Thus we obtain
\begin{align*}
&\left\{ \mathbb{E}\left[ \left( \alpha _{1\ell }\partial _{22;x}f_{\ell
}(x)-\alpha _{2\ell }\partial _{11;x}f_{\ell }(x)\right) f_{\ell }(y(\phi )
\right] \right\} ^{4}\\
&=\alpha _{1\ell }^{4}\left\{ P_{\ell }^{\prime \prime }(\cos \phi )\sin
^{2}\phi +P_{\ell }^{\prime }(\cos \phi )\cos \phi \right\} ^{4} \\
&\;\;+4\alpha _{1\ell }^{3}\alpha _{2\ell }\left\{ P_{\ell }^{\prime \prime
}(\cos \phi )\sin ^{2}\phi +P_{\ell }^{\prime }(\cos \phi )\cos \phi
\right\} ^{3}P_{\ell }^{\prime }(\cos \phi )\cos \phi \\
&\;\;+6\alpha _{1\ell }^{2}\alpha _{2\ell }^{2}\left\{ P_{\ell }^{\prime \prime
}(\cos \phi )\sin ^{2}\phi +P_{\ell }^{\prime }(\cos \phi )\cos \phi
\right\} ^{2}\left\{ P_{\ell }^{\prime }(\cos \phi )\cos \phi \right\} ^{2}
\\
&\;\;+4\alpha _{1\ell }\alpha _{2\ell }^{3}\left\{ P_{\ell }^{\prime \prime
}(\cos \phi )\sin ^{2}\phi +P_{\ell }^{\prime }(\cos \phi )\cos \phi
\right\} \left\{ P_{\ell }^{\prime }(\cos \phi )\cos \phi \right\} ^{3} \\
&\;\;+\alpha _{2\ell }^{4}\left\{ P_{\ell }^{\prime }(\cos \phi )\cos \phi
\right\} ^{4}.
\end{align*}
Now, again using Lemma \ref{AA} below,
\begin{equation*}
\int_{0}^{\pi /2}\left\{ P_{\ell }^{\prime \prime }(\cos \phi )\sin ^{2}\phi
\right\} ^{4}\sin \phi d\phi =\frac{3\ell ^{8}}{2\pi ^{2}}\frac{\log \ell }{
\ell ^{2}}+O(\ell ^{6})
\end{equation*}
and exploiting instead Lemma \ref{BB}
\begin{equation*}
\int_{0}^{\pi /2}\left\{ P_{\ell }^{\prime \prime }(\cos \phi )\sin ^{2}\phi
\right\} ^{k}\left\{ P_{\ell }^{\prime }(\cos \phi )\cos \phi \right\}
^{4-k}\sin \phi d\phi =O(\ell ^{6}), \text{ for all }k=1,\dots  4.
\end{equation*}
Noting that, for $k=1,\dots 4$,
\begin{equation*}
\alpha _{1\ell }^{4}=\frac{3^{2}}{\ell ^{8}}+O(\ell ^{-7}) \text{ and }
 \alpha _{1\ell }^{k}\alpha _{2\ell }^{4-k}=O(\ell ^{-8}),
 \end{equation*}
the proof is completed.
\end{proof}

\begin{lemma} \label{domin3}
As $ \ell \rightarrow \infty$,
\begin{equation*}
\int_{0}^{\pi /2}\mathbb{E}\left[ H_{2}(Y_{2}(\bar{x}))H_{2}(Y_{5}(\bar{x}
))H_{4}(f_{\ell }(y(\phi )))\right] \sin \phi d\phi =4!(2\cdot 3)\frac{1}{
\pi ^{2}}\frac{\log \ell }{2\ell ^{2}}+O(\ell^{-2}).
\end{equation*}
\end{lemma}

\begin{proof}
Again by Diagram Formula, we have that
\begin{align*}
&\mathbb{E}\left[ H_{2}(Y_{2}(\bar{x}))H_{2}(Y_{5}(\bar{x}))H_{4}(f_{\ell
}(y(\phi )))\right] \\
&=24\left\{ \mathbb{E}\left[ Y_{2}(\bar{x})f_{\ell
}(y(\phi ))\right] \right\}^{2} \left\{ \mathbb{E}\left[ Y_{5}(\bar{x}
)f_{\ell }(y(\phi ))\right] \right\}^{2}\\
&=24\left\{ \mathbb{E}\left[ \sqrt{\frac{2}{\lambda _{\ell }}}\partial
_{2;x}f_{\ell }(x)f_{\ell }(y(\phi )\right] \right\} ^{2}\left\{ \mathbb{E}
\left[ \alpha _{1\ell }\partial _{22;x}f_{\ell }(x)-\alpha _{2l}\partial
_{11;x}f_{\ell }(x)f_{\ell }(y(\phi )\right] \right\} ^{2} \\
&=24\frac{2}{\lambda _{\ell }}\left\{ P_{\ell }^{\prime }(\cos \phi )\sin
(\phi )\right\} ^{2}\left\{ \alpha _{1\ell }(P_{\ell }^{\prime \prime }(\cos
\phi )\sin ^{2}\phi -P_{\ell }^{\prime }(\cos \phi )\cos \phi )+\alpha
_{2l}P_{\ell }^{\prime }(\cos \phi )\cos \phi \right\} ^{2}.
\end{align*}
Now using repeatedly Lemma \ref{AA} and Lemma \ref{BB} we obtain
\begin{align*}
&\int_{0}^{\pi /2}\left\{ P_{\ell }^{\prime }(\cos \phi )\sin \phi \right\}
^{2}\left\{ \alpha _{1\ell }(P_{\ell }^{\prime \prime }(\cos \phi )\sin
^{2}\phi -P_{\ell }^{\prime }(\cos \phi )\cos \phi )+\alpha _{2l}P_{\ell
}^{\prime }(\cos \phi )\cos \phi \right\} ^{2}\sin \phi d\phi \\
&=\frac{3}{\pi ^{2}}\frac{\log \ell }{2}+O(1)\text{ ,}
\end{align*}
and thus the conclusion follows.
\end{proof}

\subsection{Subdominant terms}

The behaviour of subdominant terms can be characterized rather easily, as
follows.
\begin{lemma} \label{subdomin1}
As $\ell \rightarrow \infty$, for $a=1, 3, 4$,
\begin{equation*}
\int_{0}^{\pi /2}\mathbb{E}\left[ H_{4}(Y_{a}(\bar{x}))H_{4}(f_{\ell
}(y(\phi )))\right] \sin \phi d\phi =O(\ell ^{-2}).
\end{equation*}
\end{lemma}
\begin{proof}
For $a=1$, we have that
\begin{align*}
\mathbb{E}\left[ H_{4}(Y_{1}(\bar{x}))H_{4}(f_{\ell }(y(\phi )))\right]
&=4!\left\{ \mathbb{E}\left[ Y_{1}(\bar{x})f_{\ell }(y(\phi ))\right]
\right\} ^{4} \\
&=4!\left\{ \mathbb{E}\left[ \sqrt{\frac{2}{\ell (\ell +1)}}\partial
_{1;x}f_{\ell }(x)f_{\ell }(y(\phi )\right] \right\} ^{4} \\
&=4!\frac{2^{2}}{\ell ^{2}(\ell +1)^{2}}\left\{ \mathbb{E}\left[ \partial
_{1;x}f_{\ell }(x)f_{\ell }(y(\phi )\right] \right\} ^{4}.
\end{align*}
Now we have easily
\begin{align*}
\mathbb{E}\left[ \partial _{1;x}f_{\ell }(x)f_{\ell }(y(\phi ))\right]
&=\left. \frac{\partial }{\partial \theta }P_{\ell }(\left\langle
x,y\right\rangle )\right\vert _{x=\overline{x},y=y(\phi )} \\
&=\left. P_{\ell }^{\prime }(\left\langle x,y\right\rangle )\left\{ -\sin
\theta _{x}\cos \theta _{y}+\cos \theta _{x}\sin \theta _{y}\sin (\varphi
_{x}-\varphi _{y})\right\} \right\vert _{x=\overline{x},y=y(\phi )} \\
&=0.
\end{align*}
Similarly
\begin{equation*}
\mathbb{E}\left[ H_{4}(Y_{3}(\bar{x}))H_{4}(f_{\ell }(y(\phi )))\right] =4!
\frac{8^{2}}{\lambda _{\ell }^{2}(3\lambda _{\ell }-2)^{2}}\left\{ \mathbb{E}
\left[ \partial _{11;x}f_{\ell }(x)f_{\ell }(y(\phi )\right] \right\} ^{4},
\end{equation*}
and
\begin{align*}
&\mathbb{E}\left[ \partial _{11;x}f_{\ell }(x)f_{\ell }(y(\phi) )\right] \\
&=\left. P_{\ell }^{\prime \prime }(\left\langle x,y\right\rangle )\left\{
-\sin \theta _{x}\cos \theta _{y}+\cos \theta _{x}\sin \theta _{y}\sin
(\varphi _{x}-\varphi _{y})\right\} ^{2}\right\vert _{x=\overline{x},y=y(\phi )} \\
&=+\left. P_{\ell }^{\prime }(\left\langle x,y\right\rangle )\left\{ -\cos
\theta _{x}\cos \theta _{y}-\sin \theta _{x}\sin \theta _{y}\sin (\varphi
_{x}-\varphi _{y})\right\} ^{2}\right\vert _{x=\overline{x},y=y(\phi )} \\
&=-P_{\ell }^{\prime }(\cos \phi )\sin ^{2}\phi,
\end{align*}
whence
\begin{equation*}
\int_{0}^{\pi /2}\mathbb{E}\left[ H_{4}(Y_{a}(\bar{x}))H_{4}(f_{\ell
}(y(\phi )))\right] \sin \phi d\phi =4!\frac{8^{2}}{\lambda _{\ell
}^{2}(3\lambda _{\ell }-2)^{2}}\int_{0}^{\pi /2}\left\{ P_{\ell }^{\prime
}(\cos \phi )\sin ^{2}\phi \right\} ^{4}\sin \phi d\phi =O(\ell^{-6}).
\end{equation*}
Finally
\begin{equation*}
\mathbb{E}\left[ H_{4}(Y_{4}(\bar{x}))H_{4}(f_{\ell }(y(\phi )))\right] =4!
\frac{8^{2}}{\lambda _{\ell }^{2}(\lambda _{\ell }-2)^{2}}\left\{ \mathbb{E}
\left[ \partial _{21;x}f_{\ell }(x)f_{\ell }(y(\phi) )\right] \right\}^{4},
\end{equation*}
where
\begin{align*}
&\mathbb{E}\left[ \partial _{21;x}f_{\ell }(x)f_{\ell }(y(\phi) )\right] \\
&=\left. \frac{1}{\sin \theta _{x}}P_{\ell }^{\prime \prime }(\left\langle
x,y\right\rangle )\left\{ -\sin \theta _{x}\cos \theta _{y}+\cos \theta
_{x}\sin \theta _{y}\sin (\varphi _{x}-\varphi _{y})\right\} ^{2}\right\vert
_{x=\overline{x},y=y(\phi )} \\
&+\left. \frac{1}{\sin \theta _{x}}P_{\ell }^{\prime }(\left\langle
x,y\right\rangle )\left\{ \cos \theta _{x}\sin \theta _{y}\cos (\varphi
_{x}-\varphi _{y})\right\} ^{2}\right\vert _{x=\overline{x},y=y(\phi )} \\
&=0.
\end{align*}
\end{proof}

\begin{lemma}
\label{subdomin2} For $a=1,4,$ we have that
\begin{equation*}
\int_{0}^{\pi /2}\mathbb{E}\left[ H_{2}(Y_{a}(\bar{x}))H_{2}(Y_{c}(\bar{x}
))H_{4}(f_{\ell }(y(\phi )))\right] \sin \phi d\phi =0, \;\; \text{ where }c=1,\dots 5,
\text{ }c\neq a.
\end{equation*}
\end{lemma}
\begin{proof}
It was shown in the proof of Lemma \ref{subdomin1} that $\mathbb{E}\left[
Y_{1}(\bar{x}))f_{\ell }(y(\phi )\right] =\mathbb{E}\left[ Y_{4}(\bar{x}
))f_{\ell }(y(\phi )\right] =0.$ The result is then an immediate consequence
of the Diagram Formula.
\end{proof}
\noindent We are then left with only two terms to consider.
\begin{lemma}
\label{subdomin3}For $a=2,5,$ we have that
\begin{equation*}
\int_{0}^{\pi /2}\mathbb{E}\left[ H_{2}(Y_{a}(\bar{x}))H_{2}(Y_{3}(\bar{x}
))H_{4}(f_{\ell }(y(\phi )))\right] \sin \phi d\phi =O(\ell^{-2}).
\end{equation*}
\end{lemma}
\begin{proof}
We have that
\begin{align*}
&\mathbb{E}\left[ H_{2}(Y_{2}(\bar{x}))H_{2}(Y_{3}(\bar{x}))H_{4}(f_{\ell
}(y(\phi )))\right] \\
&=4!\left\{ \mathbb{E}\left[ Y_{2}(\bar{x})f_{\ell }(y(\phi ))\right]
\right\} ^{2}\left\{ \mathbb{E}\left[ Y_{3}(\bar{x})f_{\ell }(y(\phi ))\right]
\right\} ^{2} \\
&=4!\times \frac{2}{\ell (\ell +1)}\left\{ \mathbb{E}\left[ \partial
_{2;x}f_{\ell }(x)f_{\ell }(y(\phi ))\right] \right\} ^{2}\times \frac{8}{
\lambda _{\ell }(3\lambda _{\ell }-2)}\left\{ \mathbb{E}\left[ \partial
_{11;x}f_{\ell }(x)f_{\ell }(y(\phi) )\right] \right\} ^{2} \\
&=4!\times \frac{2}{\ell (\ell +1)}\left\{ P_{\ell }^{\prime }(\cos \phi
)\sin (\phi )\right\} ^{2}\times \frac{8}{\lambda _{\ell }(3\lambda _{\ell
}-2)}\left\{ P_{\ell }^{\prime }(\cos \phi )\sin ^{2}\phi \right\}^{2},
\end{align*}
and therefore
\begin{align*}
\int_{0}^{\pi /2}\mathbb{E}\left[ H_{2}(Y_{2}(\bar{x}))H_{2}(Y_{3}(\bar{x}
))H_{4}(f_{\ell }(y(\phi )))\right] \sin \phi d\phi
&\leq \text{const} \times \frac{1}{\ell ^{6}}\int_{0}^{\pi /2}\left\{ P_{\ell
}^{\prime }(\cos \phi )\sin (\phi )\right\} ^{4}\sin ^{3}\phi d\phi \\&=O(\ell^{-4}).
\end{align*}
Finally
\begin{align*}
\mathbb{E}\left[ H_{2}(Y_{5}(\bar{x}))H_{2}(Y_{3}(\bar{x}))H_{4}(f_{\ell
}(y(\phi )))\right]&=4!\left\{ \mathbb{E}\left[ Y_{5}(\bar{x})f_{\ell
}(y(\phi) )\right] \right\} ^{2}\left\{ \mathbb{E}\left[ Y_{3}(\bar{x}
)f_{\ell }(y(\phi ))\right] \right\} ^{2}\\
&=4!\left\{ \mathbb{E}\left[ (\alpha _{1\ell }\partial _{22;x}f_{\ell
}(x)-\alpha _{2\ell }\partial _{11;x}f_{\ell }(x))f_{\ell }(y(\phi) )\right]
\right\} ^{2} \\
&\;\;\times \frac{8}{\lambda _{\ell }(3\lambda _{\ell }-2)}\left\{ \mathbb{E}
\left[ \partial _{11;x}f_{\ell }(x)f_{\ell }(y(\phi) )\right] \right\} ^{2} \\
&=4!\left\{ \alpha _{1\ell }(P_{\ell }^{\prime \prime }(\cos \phi )\sin
^{2}(\phi )-P_{\ell }^{\prime }(\cos \phi )\cos \phi )+\alpha _{2\ell
}P_{\ell }^{\prime }(\cos \phi )\cos \phi \right\} ^{2} \\
&\;\; \times \frac{8}{\lambda _{\ell }(3\lambda _{\ell }-2)}\left\{ P_{\ell
}^{\prime }(\cos \phi )\cos \phi \right\} ^{2};
\end{align*}
exploiting again Lemma \ref{BB}, the result follows.
\end{proof}

\begin{lemma} \label{odd1}
As $\ell \to \infty$, for $a, b=3, 5$, $a \ne b$
\begin{align*}
\int_{0}^{\pi / 2} \mathbb{E}[H_{3}(Y_{a}(\bar{x}))H_{1}(Y_{b}(\bar{x})) H_4(f_{\ell}(y(\phi)))] \sin \phi d \phi= O(\ell^{-2}).
\end{align*}
\end{lemma}
\begin{proof}
By applying again the Diagram Formula (see \cite{MaPeCUP} Section 4.3.1), we have
\begin{align*}
&\int_{0}^{\pi / 2} \mathbb{E}[H_{3}(Y_{a}(\bar{x}))H_{1}(Y_{b}(\bar{x})) H_4(f_{\ell}(y(\phi)))] \sin \phi d \phi \\
& =  \int_{0}^{\pi / 2}\{ 3^2 4 \mathbb{E}[ Y_{a}(\bar{x}) Y_{b}(\bar{x}) ]  \mathbb{E}^2[ Y_{a}(\bar{x}) f_{\ell}(y(\phi))  ]+ 4!  \mathbb{E}^3[ Y_{a}(\bar{x}) f_{\ell}(y(\phi))  ]   \mathbb{E}[ Y_{b}(\bar{x}) f_{\ell}(y(\phi))  ] \} \sin \phi d \phi.
\end{align*}
We observe that
 \begin{align*}
\mathbb{E}[ Y_{3}(\bar{x}) Y_{5}(\bar{x}) ]&=   \frac{\sqrt 8 }{\sqrt{\lambda_{\ell}} \sqrt{3 \lambda_{\ell}-2}} \frac{\sqrt{3 \lambda_{\ell}-2}}{\lambda_{\ell} \sqrt{\lambda_{\ell}-2}}  \mathbb{E}[ \partial _{22;x}f_{\ell }(\bar{x}) \partial_{11;x}f_{\ell }(\bar{x}) ]\\
&\;\; -\frac{\sqrt 8 }{\sqrt{\lambda_{\ell}} \sqrt{3 \lambda_{\ell}-2} } \frac{\lambda_{\ell}+2}{\lambda_{\ell} \sqrt{3 \lambda_{\ell}-2} \sqrt{\lambda_{\ell}-2}}  \mathbb{E}[ \partial_{11;x}f_{\ell }(\bar{x}) \partial_{11;x}f_{\ell }(\bar{x})]\\
& =   \frac{\sqrt 8 }{\sqrt{\lambda_{\ell}} \sqrt{3 \lambda_{\ell}-2}} \frac{\sqrt{3 \lambda_{\ell}-2}}{\lambda_{\ell} \sqrt{\lambda_{\ell}-2}} \frac{\lambda_{\ell}}{8}[  \lambda_{\ell}+ 2] -\frac{\sqrt 8 }{\sqrt{\lambda_{\ell}} \sqrt{3 \lambda_{\ell}-2} } \frac{\lambda_{\ell}+2}{\lambda_{\ell} \sqrt{3 \lambda_{\ell}-2} \sqrt{\lambda_{\ell}-2}}   \frac{\lambda_{\ell}}{8}[ 3 \lambda_{\ell} -2]\\
& = 0,
\end{align*}
moreover
\begin{align*}
 \mathbb{E}[ Y_{3}(\bar{x}) f_{\ell}(y(\phi))  ] &= - \frac{\sqrt{8}}{\sqrt{\lambda_{\ell}} \sqrt{3 \lambda_{\ell}-2}}   P'_{\ell}(\cos \phi) \cos \phi,
\end{align*}
and
\begin{align*}
\mathbb{E}[ Y_{5}(\bar{x}) f_{\ell}(y(\phi))  ]  &= \frac{\sqrt{3 \lambda_{\ell}-2}}{\lambda_{\ell} \sqrt{ \lambda_{\ell}-2}}  \mathbb{E}[ \partial _{22;x}f_{\ell }(\bar{x}) f_{\ell}(y(\phi))]-\frac{\lambda_{\ell}+2}{\lambda_{\ell} \sqrt{ \lambda_{\ell}-2} \sqrt{3 \lambda_{\ell}-2}}  \mathbb{E}[ \partial_{11;x}f_{\ell }(\bar{x})f_{\ell}(y(\phi))]\\
&=\frac{\sqrt{3 \lambda_{\ell}-2}}{\lambda_{\ell} \sqrt{ \lambda_{\ell}-2}}  [P''(\cos \phi) \sin^2 \phi - P'_{\ell}(\cos \phi) \cos \phi ]-\frac{\lambda_{\ell}+2}{\lambda_{\ell} \sqrt{ \lambda_{\ell}-2} \sqrt{3 \lambda_{\ell}-2}}  [- P'_{\ell}(\cos \phi) \cos \phi ].
\end{align*}
The statement follows by applying Lemma \ref{BB}.
\end{proof}

\begin{lemma} \label{odd2}
For $a=1, 2$,
\begin{align*}
\int_{0}^{\pi / 2} \mathbb{E}[H_{2}(Y_{a}(\bar{x})) H_{1}(Y_{3}(\bar{x})) H_{1}(Y_{5}(\bar{x}))  H_4(f_{\ell}(y(\phi)))] \sin \phi d \phi=0.
\end{align*}
\end{lemma}
\begin{proof}
From Diagram Formula we have
\begin{align*}
& \int_{0}^{\pi / 2} \mathbb{E}[H_{2}(Y_{a}(\bar{x})) H_{1}(Y_{3}(\bar{x})) H_{1}(Y_{5}(\bar{x}))  H_4(f_{\ell}(y(\phi)))] \sin \phi d \phi \\
 &=   \int_{0}^{\pi / 2}\{  2 \mathbb{E}[Y_{a}(\bar{x}) Y_{3}(\bar{x}) ]  \mathbb{E}[ Y_{a}(\bar{x}) Y_{5}(\bar{x})] + 4! \mathbb{E}[Y_{a}(\bar{x}) Y_{3}(\bar{x})] \mathbb{E}[Y_{3}(\bar{x})f_{\ell}(y(\phi))] \mathbb{E}[Y_{a}(\bar{x})f_{\ell}(y(\phi))] \\
 &\;\;+ 3 \cdot 4 \mathbb{E}[Y_{3}(\bar{x}) Y_{5}(\bar{x})]\mathbb{E}^2[Y_{a}(\bar{x}) f_{\ell}(y(\phi))] + 4! \mathbb{E}[Y_{a}(\bar{x}) Y_{5}(\bar{x})]\mathbb{E}[Y_{3}(\bar{x}) f_{\ell}(y(\phi))]\mathbb{E}[Y_{a}(\bar{x}) f_{\ell}(y(\phi))] \\
 &\;\;+ 4! \mathbb{E}[Y_{3}(\bar{x}) f_{\ell}(y(\phi))]\mathbb{E}[Y_{5}(\bar{x}) f_{\ell}(y(\phi))]\mathbb{E}^2[Y_{a}(\bar{x}) f_{\ell}(y(\phi))] \} \sin \phi d \phi.
\end{align*}
The statement follows by observing that (for $a=1,2$)  $\mathbb{E}[Y_{a}(\bar{x}) Y_{3}(\bar{x}) ]=0$, $\mathbb{E}[ Y_{a}(\bar{x}) Y_{5}(\bar{x})]=0$, $\mathbb{E}[ Y_{3}(\bar{x}) Y_{5}(\bar{x}) ]=0$, and $\mathbb{E}[Y_{a}(\bar{x})f_{\ell}(y(\phi))] =0$.
\end{proof}

\begin{lemma} \label{odd3}
We have
\begin{align*}
 \int_{0}^{\pi / 2} \mathbb{E}[H_{2}(Y_{4}(\bar{x})) H_{1}(Y_{3}(\bar{x}))  H_{1}(Y_{5}(\bar{x})) H_4(f_{\ell}(y(\phi)))] \sin \phi d \phi=0.
 \end{align*}
\end{lemma}
\begin{proof}
Once again, the statement follows from Diagram Formula, which gives
\begin{align*}
 &\int_{0}^{\pi / 2} \mathbb{E}[H_{2}(Y_{4}(\bar{x})) H_{1}(Y_{3}(\bar{x}))  H_{1}(Y_{5}(\bar{x})) H_4(f_{\ell}(y(\phi)))] \sin \phi d \phi \\
 &=   \int_{0}^{\pi / 2}\{  2 \mathbb{E}[Y_{4}(\bar{x}) Y_{3}(\bar{x}) ]  \mathbb{E}[ Y_{4}(\bar{x}) Y_{5}(\bar{x})] + 4! \mathbb{E}[Y_{4}(\bar{x}) Y_{3}(\bar{x})] \mathbb{E}[Y_{3}(\bar{x})f_{\ell}(y(\phi))] \mathbb{E}[Y_{4}(\bar{x})f_{\ell}(y(\phi))] \\
 &\;\;+ 3 \cdot 4 \mathbb{E}[Y_{3}(\bar{x}) Y_{5}(\bar{x})]\mathbb{E}^2[Y_{4}(\bar{x}) f_{\ell}(y(\phi))] + 4! \mathbb{E}[Y_{4}(\bar{x}) Y_{5}(\bar{x})]\mathbb{E}[Y_{3}(\bar{x}) f_{\ell}(y(\phi))]\mathbb{E}[Y_{4}(\bar{x}) f_{\ell}(y(\phi))] \\
 &\;\;+ 4! \mathbb{E}[Y_{3}(\bar{x}) f_{\ell}(y(\phi))]\mathbb{E}[Y_{5}(\bar{x}) f_{\ell}(y(\phi))]\mathbb{E}^2[Y_{4}(\bar{x}) f_{\ell}(y(\phi))] \} \sin \phi d \phi
\end{align*}
and $\mathbb{E}[Y_{4}(\bar{x}) Y_{3}(\bar{x}) ]=0$, $\mathbb{E}[ Y_{4}(\bar{x}) Y_{5}(\bar{x})]=0$, $\mathbb{E}[ Y_{3}(\bar{x}) Y_{5}(\bar{x}) ]=0$, $\mathbb{E}[Y_{4}(\bar{x})f_{\ell}(y(\phi))] =0$.
\end{proof}

\subsection{Some useful integrals}
We write as usual
\begin{equation*}
P_{\ell }^{(r)}(u)=\frac{d^{r}}{du^{r}}P_{\ell }(u).
\end{equation*}
For our main arguments to follow, a key step is to recall the following
results, which are proved in \cite{CMW}, Lemma C3.
For all constants $C>0,$ we have, uniformly over ${C}/{\ell }\leq \phi
\leq {\pi }/{\ell }$
\begin{equation} \label{C3a}
P_{\ell }^{(r)}(u)=\sqrt{\frac{2}{\pi }}\frac{\ell ^{2r-\frac{1}{2}}}{\sin
^{r+\frac{1}{2}}\phi }(-1)^{r/2}\cos \psi _{\ell }^{\pm }+R_{\ell
}^{(r)}(\phi ), \hspace{1cm} r=0,1,2,
\end{equation}
where $\psi _{\ell }^{\pm }=(\ell +{1}/{2})\phi -{\pi }/{4}$ for $
r=0, 2$, and $\psi _{\ell }^{\pm }=(\ell +{1}/{2})\phi +{\pi }/{4}$
for $r=1$, and
\begin{equation}
R_{\ell }^{(0)}(\phi )=O\left(\frac{1}{\sqrt{\ell \phi} }\right), \;\; R_{\ell }^{(1)}(\phi )=O\left(\frac{1}{\sqrt{\ell }\phi^{5/2}}\right), \;\; R_{\ell }^{(2)}(\phi )=O\left(\frac{\sqrt{\ell }}{\phi ^{7/2}}\right).
\label{C3b}
\end{equation}
Our results will then follow from the following two lemmas:
\begin{lemma} \label{AA}
For $r=0, 1, 2$, we have
\begin{equation*}
\int_{0}^{\pi /2}\;[P_{\ell }^{(r)}(\cos \phi )\sin ^{r}\phi ]^{4}\sin \phi
d\phi =\frac{3\ell ^{4r}}{\pi ^{2}}\frac{\log \ell }{2\ell ^{2}}+O(\ell
^{4r-2}).
\end{equation*}
Likewise
\begin{align*}
&\int_{0}^{\pi /2}\;[P_{\ell }(\cos \phi )]^{2}[P_{\ell }^{\prime \prime
}(\cos \phi )\sin ^{2}\phi ]^{2}\sin \phi d\phi =\frac{3\ell ^{2}\log \ell }{
2\pi ^{2}}+O(\ell ^{2}),\\
&\int_{0}^{\pi /2}\;[P_{\ell }(\cos \phi )]^{2}[P_{\ell }^{\prime }(\cos \phi
)\sin \phi ]^{2}\sin \phi d\phi =\frac{\log \ell }{2\pi ^{2}}+O(1),\\
&\int_{0}^{\pi /2}\;[P_{\ell }^{\prime }(\cos \phi )]^{2}[P_{\ell }^{\prime
\prime }(\cos \phi )\sin \phi ]^{2}\sin \phi d\phi =\frac{\ell ^{4}\log \ell
}{2\pi ^{2}}+O(\ell ^{4}).
\end{align*}
\end{lemma}
\begin{remark}
More compactly, for $r_{1}, r_{2}=0, 1, 2$, we could have written the single expression
\begin{align*}
\int_{0}^{\pi /2}\;[P_{\ell }^{(r_{1})}(\cos \phi )\sin ^{r_{1}}\phi
]^{2}[P_{\ell }^{(r_{2})}(\cos \phi )\sin ^{r_{2}}\phi ]^{2}\sin \phi d\phi &=
\frac{(2+(-1)^{r_{1}+r_{2}})\ell ^{2(r_{1}+r_{2})}}{\pi ^{2}}\frac{\log \ell
}{2\ell ^{2}}\\
&\;\;+O(\ell ^{2(r_{1}+r_{2})-2}).
\end{align*}
\end{remark}
\begin{proof}
We recall first that $P_{\ell }^{(r)}(\cos \phi )\leq \ell ^{2r}$ for all $\phi \in \lbrack
0,2\pi )$. Hence
\begin{equation*}
\int_{0}^{C/\ell }[P_{\ell }^{(r)}(\cos \phi )\sin ^{r}\phi ]^{4}\sin \phi
d\phi \leq \text{const} \times \ell ^{8r}\int_{0}^{C/\ell }\sin ^{4r+1}\phi d\phi =O(\ell ^{4r-2}),
\end{equation*}
and it suffices to consider $\phi >{C}/{\ell }$. Hence we have
\begin{align*}
&\int_{0}^{\pi /2}[P_{\ell }^{(r)}(\cos \phi )\sin ^{r}\phi ]^{4}\sin \phi
d\phi\\
&=\frac{2^{2}}{\pi ^{2}}\int_{C/\ell }^{\pi /2}\left[ \frac{\ell ^{r}}{
\sqrt{\ell \sin \phi }}\cos \left((\ell +{1}/{2})\phi \pm \frac{\pi }{4}\right)
\right] ^{4}\sin \phi d\phi  \\
&\;\;+4\frac{2^{3/2}}{\pi ^{3/2}}\int_{C/\ell }^{\pi /2}\left[ \frac{\ell ^{r}}{
\sqrt{\ell \sin \phi }}\cos \left((\ell +{1}/{2})\phi \pm \frac{\pi }{4}\right)
\right] ^{3}\left[ R_{\ell }^{(r)}(\phi )\sin ^{r}\phi \right] \sin \phi
d\phi  \\
&\;\;+6\frac{2}{\pi }\int_{C/\ell }^{\pi /2}\left[ \frac{\ell ^{r}}{\sqrt{\ell
\sin \phi }}\cos \left((\ell +{1}/{2})\phi \pm \frac{\pi }{4}\right)\right] ^{2}
\left[ R_{\ell }^{(r)}(\phi )\sin ^{r}\phi \right] ^{2}\sin \phi d\phi  \\
&\;\;+4\frac{2^{1/2}}{\pi ^{1/2}}\int_{C/\ell }^{\pi /2}\left[ \frac{\ell ^{r}}{
\sqrt{\ell \sin \phi }}\cos \left((\ell +{1}/{2})\phi \pm \frac{\pi }{4}\right)
\right] \left[ R_{\ell }^{(r)}(\phi )\sin ^{r}\phi \right] ^{3}\sin \phi
d\phi  \\
&\;\;+\int_{C/\ell }^{\pi /2}\left[ R_{\ell }^{(r)}(\phi )\sin ^{r}\phi \right]^{4}\sin \phi d\phi.
\end{align*}
It is not difficult to see that, for $k=1,\dots ,4$,
\begin{equation*}
\int_{C/\ell }^{\pi /2}\left[ \frac{\ell ^{r}}{\sqrt{\ell \sin \phi }}\cos
\left((\ell +{1}/{2})\phi \pm \frac{\pi }{4}\right)\right] ^{k}\left[ R_{\ell
}^{(r)}(\phi )\sin ^{r}\phi \right] ^{4-k}\sin \phi d\phi =O(\ell ^{4r-2});
\end{equation*}
indeed the previous integrals are bounded by, for $r=2$,
\begin{align*}
\ell ^{kr-k/2}\int_{C/\ell }^{\pi /2}\frac{1}{\sin ^{k/2}\phi }\left[
R_{\ell }^{(2)}(\phi )\sin ^{r}\phi \right] ^{4-k}\sin \phi d\phi  &\leq
 \text{const} \times \ell ^{3k/2}\int_{C/\ell }^{\pi /2}\frac{1}{\sin ^{k/2}\phi }
\left[ \frac{\ell ^{1/2}}{\phi ^{3/2}}\right] ^{4-k}\sin \phi d\phi  \\
&\leq   \text{const} \times \ell ^{k+2}\int_{C/\ell }^{\pi /2}\frac{1}{\sin
^{k/2}\phi }\left[ \frac{1}{\phi ^{3/2}}\right] ^{4-k}\sin \phi d\phi  \\
&\leq   \text{const} \times \ell ^{k+2}\int_{C/\ell }^{\pi /2}\frac{1}{\sin
^{k/2}\phi }\left[ \frac{1}{\phi ^{3/2}}\right] ^{4-k}\sin \phi d\phi  \\
&\leq  \text{const} \times \ell ^{k+2}\int_{C/\ell }^{\pi /2}\frac{1}{\sin
^{k/2}\phi }\frac{1}{\phi ^{6-3k/2}}\sin \phi d\phi  \\
&\leq  \text{const} \times \ell ^{k+2}\int_{C/\ell }^{\pi /2}\phi ^{k-5}d\phi
=O(\ell ^{6}).
\end{align*}
Likewise, for $r=1$,
\begin{align*}
\ell ^{k-k/2}\int_{C/\ell }^{\pi /2}\frac{1}{\sin ^{k/2}\phi }\left[ R_{\ell
}^{(r)}(\phi )\sin ^{r}\phi \right] ^{4-k}\sin \phi d\phi  &\leq
 \text{const} \times \ell ^{k/2}\int_{C/\ell }^{\pi /2}\frac{1}{\sin ^{k/2}\phi }
\left[ \frac{1}{\ell ^{1/2}\phi ^{3/2}}\right] ^{4-k}\sin \phi d\phi  \\
&\leq  \text{const} \times \ell ^{k-2}\int_{C/\ell }^{\pi /2}\frac{1}{\sin
^{k/2}\phi }\left[ \frac{1}{\phi ^{3/2}}\right] ^{4-k}\sin \phi d\phi  \\
&\leq  \text{const} \times \ell ^{k-2}\int_{C/\ell }^{\pi /2}\frac{1}{\sin
^{k/2}\phi }\frac{1}{\phi ^{6-3k/2}}\sin \phi d\phi  \\
&\leq  \text{const} \times \ell ^{k-2}\int_{C/\ell }^{\pi /2}\phi ^{k-5}d\phi
=O(\ell ^{2}).
\end{align*}
Thus
\begin{equation*}
\int_{0}^{\pi /2}[P_{\ell }^{(r)}(\cos \phi )\sin ^{r}\phi ]^{4}\sin \phi
d\phi =\frac{2^{2}}{\pi ^{2}}\int_{C/\ell }^{\pi /2}\left[ \frac{\ell ^{r}}{
\sqrt{\ell \sin \phi }}\cos ((\ell +\frac{1}{2})\phi \pm \frac{\pi }{4})
\right] ^{4}\sin \phi d\phi +O(\ell ^{4r-2}).
\end{equation*}
The following equalities can be established by simple trigonometric
identities:
\begin{equation*}
\cos ^{4}((\ell +\frac{1}{2})\phi -\frac{\pi }{4})=\frac{3}{8}+\frac{1}{8}
(-\cos (2\phi (2\ell +1))+4\sin (\phi (2\ell +1))),
\end{equation*}
\begin{equation*}
\cos ^{4}((\ell +\frac{1}{2})\phi +\frac{\pi }{4})=\frac{3}{8}+\frac{1}{8}
(-\cos (2\phi (2\ell +1))-4\sin (\phi (2\ell +1))).
\end{equation*}
Thus we have
\begin{align*}
\int_{0}^{\pi /2}[P_{\ell }^{(r)}(\cos \phi )\sin ^{r}\phi ]^{4}\sin \phi
d\phi & =\ell ^{4r-2}\frac{2^{2}}{\pi ^{2}}\frac{3}{8}\int_{C/\ell }^{\pi /2}
\frac{1}{\sin \phi }d\phi +O(\ell ^{4r-2}) \\
& =\ell ^{4r-2}\frac{2^{2}}{\pi ^{2}}\frac{3}{8}\log \ell ++O(\ell ^{4r-2})\\
& =\frac{3}{2\pi ^{2}}\ell ^{4r-2}\log \ell +O(\ell ^{4r-2}),
\end{align*}
since
\begin{equation*}
\int_{C/\ell }^{\pi /2}\frac{1}{\sin \phi }d\phi =\left. \frac{1}{2}\log
\left( \frac{1-\cos \phi }{1+\cos \phi }\right) \right\vert _{C/\ell }^{\pi
/2}=\log \ell +O(1).
\end{equation*}
The proof of the first part of the lemma is then concluded. The proof of the
second result is very similar and we can omit some details; in particular,
we simply recall the identity
\begin{align*}
&\cos ^{2}\left(\frac{2\ell +1}{2}\phi +\frac{\pi }{4}\right)\cos^{2}\left(\frac{2\ell +1}{
2}\phi -\frac{\pi }{4}\right) \\
&=\left[\frac{\sqrt{2}}{2}\cos \left(\frac{2\ell +1}{2}\phi \right)-\frac{\sqrt{2}}{2}\sin \left(
\frac{2\ell +1}{2}\phi \right)\right]^{2}\left[\frac{\sqrt{2}}{2}\cos \left(\frac{2\ell +1}{2}\phi
\right)+\frac{\sqrt{2}}{2}\sin \left(\frac{2\ell +1}{2}\phi \right)\right]^{2} \\
&=\frac{1}{4}\left[\cos^{2}\left(\frac{2\ell +1}{2}\phi \right)-\sin^{2}\left(\frac{2\ell +1}{2
}\phi \right)\right]^{2} \\
&=\frac{1}{4}\cos ^{2}((2\ell +1)\phi ).
\end{align*}
Because $\cos 2x+1=2\cos ^{2}x$, it is not difficult to see that
\begin{align*}
\int_{C/\ell }^{\pi /2}\frac{\cos ^{2}((2\ell +1)\phi )}{\sin \phi }d\phi
&=\frac{1}{2}\int_{C/\ell }^{\pi /2}\frac{1}{\sin \phi }d\phi +\int_{C/\ell
}^{\pi /2}\frac{\cos (2(2\ell +1)\phi )}{2\sin \phi }d\phi  =\frac{1}{2}\log \ell +O(1).
\end{align*}
Dealing with the lower order terms as in the first part of the lemma we can
now conclude with our second statement, i.e.,
\begin{align*}
\int_{0}^{\pi /2}\;[P_{\ell }(\cos \phi )\sin ^{r_{1}}\phi ]^{2}[P_{\ell
}^{(4)}(\cos \phi )\sin ^{r_{2}}\phi ]^{2}\sin \phi d\phi  &=\frac{3\ell
^{8}}{2\pi ^{2}}\frac{\log \ell }{\ell ^{2}}+O(\ell ^{6}),  \\
\int_{0}^{\pi /2}\;[P_{\ell }(\cos \phi )\sin ^{r_{1}}\phi ]^{2}[P_{\ell
}^{(2)}(\cos \phi )\sin ^{r_{2}}\phi ]^{2}\sin \phi d\phi  &=\frac{\ell ^{8}
}{\pi ^{2}}\frac{\log \ell }{2\ell ^{2}}+O(\ell ^{6}), \\
\int_{0}^{\pi /2}\;[P_{\ell }^{(2)}(\cos \phi )\sin ^{r_{1}}\phi
]^{2}[P_{\ell }^{(4)}(\cos \phi )\sin ^{r_{2}}\phi ]^{2}\sin \phi d\phi  &=
\frac{\ell ^{8}}{\pi ^{2}}\frac{\log \ell }{2\ell ^{2}}+O(\ell ^{6}).
\end{align*}
\end{proof}
\noindent In our second auxiliary result, an upper bound is given.
\begin{lemma} \label{BB} As $\ell \rightarrow \infty$, we have that
\begin{equation*}
\int_{0}^{\pi /2}\left\vert P_{\ell }^{\prime }(\cos \phi )\right\vert
^{k}\left\vert P_{\ell }^{\prime \prime }(\cos \phi )\sin ^{2}\phi
\right\vert ^{4-k}\sin \phi d\phi =O(\ell ^{6}), \;\; \text{for all}\;\;  k=1,\dots 4.
\end{equation*}
\end{lemma}
\begin{proof}
As before, for the ``local" component where $\phi <C/\ell $ (some fixed constant $C)$ we have
\begin{align*}
\int_{0}^{C/\ell }\;\left\{ P_{\ell }^{\prime \prime }(\cos \phi )\sin
^{2}\phi \right\} ^{k}\left\{ P_{\ell }^{\prime }(\cos \phi )\cos \phi
\right\} ^{4-k}\sin \phi  &\leq  \text{const} \times \ell ^{4k}\times \ell
^{8-2k}\int_{0}^{C/\ell }\;\sin ^{2k}\phi \sin \phi d\phi  \\
&=O(\ell ^{8+2k-(2k+2)})=O(\ell ^{6})\ .
\end{align*}
On the other hand, using again formulas \eqref{C3a} and \eqref{C3b}, and computations
analogous to Lemma \ref{AA}, we find easily that
\begin{align*}
&\int_{C/\ell }^{\pi /2}\left\{ P_{\ell }^{\prime }(\cos \phi )\cos \phi
\right\} ^{k}\left\{ P_{\ell }^{\prime \prime }(\cos \phi )\sin ^{2}\phi
\right\} ^{4-k}\sin \phi d\phi  \\
&\leq  \text{const} \times \ell ^{k/2}\times \ell ^{6-3k/2}\int_{C/\ell }^{\pi /2}
\frac{1}{\sin ^{3k/2}\phi }\frac{1}{\sin ^{2-k/2}\phi }\sin \phi d\phi
+O(\ell ^{6}) \\
&\leq  \text{const} \times \ell ^{6-k}\int_{C/\ell }^{\pi /2}\frac{1}{\sin
^{2+k}\phi }\sin \phi d\phi +O(\ell ^{6})=O(\ell ^{6}).
\end{align*}
\end{proof}
\noindent Note in particular that for $k=4$ we obtain the bound
\begin{equation*}
\int_{0}^{\pi /2}[P_{\ell }^{\prime }(\cos \phi )]^{4}\sin \phi d\phi
=O(\ell^{6}).
\end{equation*}

\end{document}